\documentclass[12pt, a4paper]{article}

\usepackage{amsmath, amsthm, amssymb}
\usepackage{fancyhdr}
\usepackage{graphicx}
\usepackage{enumerate}
\usepackage[top=2.5cm, bottom=2.5cm, left=2.5cm, right=2.5cm]{geometry}

\newcommand{\x}{\mathbf x}
\newcommand{\y}{\mathbf y}
\newcommand{\z}{\mathbf z}
\newcommand{\h}{\mathbf h}
\newcommand{\vv}{\mathbf v}

\newtheorem{definition}{Definition}
\newtheorem{lemma}{\bf Lemma}

\newtheorem{theorem}{\bf Theorem}
\newtheorem{corollary}{\bf Corollary}

\renewenvironment{proof}{\noindent {\bf Proof: }}{\rm\\}
\theoremstyle{definition}
\newtheorem{remark}{Remark}{\rm}
{\rm}

\usepackage{algorithm}
\usepackage{algorithmic}
\usepackage[T1]{fontenc}

\usepackage{mathrsfs}

\usepackage{color}

\usepackage{algorithm, algorithmic}

\makeatletter
\renewcommand{\p@algorithm}{\arabic{algorithm}\expandafter\@gobble}
\makeatother 
\newcommand{\PARAMETERS}{\item[\textbf{Parameters:}]}

\newcounter{step}[algorithm]
\newcommand\STEP[2][\(\triangleright\)]{%
	\refstepcounter{step}
	\vskip 0.25\baselineskip
	\item[]\hskip -\algorithmicindent #1 \textbf{Step \arabic{step}}%
	\ifthenelse{\equal{\unexpanded{#2}}{}}{}{ (\texttt{#2})}%
	\textbf{.}%
}
\newenvironment{algo}{\algo}{}
\def\algo#1\end{%
	\noindent\fbox{%
	\begin{minipage}[b]{\dimexpr\columnwidth-\algorithmicindent\relax}
	\begin{algorithmic}
	#1
	\end{algorithmic}
	\end{minipage}
	}%
\end}
\begin{document}

\title{Non-smooth optimization for robust control of infinite-dimensional systems\thanks{Dedicated to the memory of Jon Borwein}}
\author{Pierre Apkarian
\thanks{ONERA, Control System Department, Toulouse, France}
\and
{Dominikus Noll}
\thanks{Universit\'e de Toulouse, Institut de Math\'ematiques, Toulouse, France}
\and
{Laleh Ravanbod}$\,^\ddag$
%\thanks{Universit\'e de Toulouse, Institut de Math\'ematiques, Toulouse, France}
}
\date{}

\maketitle

\begin{abstract}
We use a non-smooth trust-region method for $H_\infty$-control of infinite-dimensional systems. Our method applies in particular
to distributed and boundary control 
of partial differential equations. It is computationally
attractive as it avoids the use of system reduction or identification. For illustration the method is applied to control a reaction-convection-diffusion system,
a Van de Vusse reactor, and to a cavity flow control problem.

\vspace{.2cm}
\noindent {\bf Keywords:}
Robust control $\cdot$ non-smooth optimization $\cdot$ non-smooth trust-region method $\cdot$ infinite-dimensional system $\cdot$ boundary control
\end{abstract}

\section{Introduction}
Feedback control
of partial differential equations  and other infinite-dimensional systems encounters limitations due to computational issues. In state-space,
PDE models are 
overly complex and not directly  suited for controller synthesis.   
System reduction is required to 
bring the state-space down to a size where synthesis
methods are applicable.  Not only is this technically demanding, it also bears
the risk of producing inaccurate or oversimplified models, where  $H_\infty$-performance can no longer be guaranteed.

Computing the system transfer function directly from the infinite dimensional model
avoids this loss  of information, but encounters a second difficulty. Customary
strategies now try to fit a finite-dimensional state-space model  to the infinite-dimensional transfer function. This uses
optimization-based identification techniques, which are
in conflict with the $H_\infty$-objective, as the two optimization procedures in series are no longer meaningful. 
In addition, for unstable systems the identification often uses heuristics or ad hoc approaches, which have no certificates.

The method we propose here avoids both pitfalls. We synthesize controllers directly from the pre-computed frequency response, thereby avoiding 
system reduction and identification. Discretization for computation is performed in frequency space on a low-dimensional object, 
which avoids the loss of information.
Our tests demonstrate that this works fast and reliably, once the transfer function is available. It turns out that the success of our method hinges on 
the use of non-smooth
optimization. We use a non-smooth trust-region method first proposed in \cite{ANR:2016},  which
allows trial steps tailored to the specific application. We prove convergence under Kiwiel's aggregation rule, a question which had remained open in \cite{ANR:2016}.
This given an affirmative answers to a question already posed in \cite{rus} for the convex non-smooth trust-region method. 
For  complementary information on bundle methods see \cite{saga,hare,noll}, a mix of bundle and trust-regions is \cite{zowe}.

Design of controllers
in the frequency domain based on non-smooth optimization has already been performed in \cite{polak1,polak2,polak3,polak4,kiwiel_automatic,cullum}. Structured
$H-\infty$-control for infinite-dimensional systems is addressed in \cite{IJRNC}.
$H_\infty$-control of a heat exchange system is discussed in \cite{sano}.  These approaches use either unstructured controllers, are based on  matrix inequalities, or 
differ with regard to the optimization technique.

The structure of the paper is as follows. In section \ref{strategy} we outline our approach to $H_\infty$-control of
infinite dimensional systems.
In section \ref{optimize} we discuss optimization and present our non-smooth trust-region method originally proposed
in \cite{ANR:2016},  on which the present approach rests. 
Convergence of the non-smooth trust-region method is discussed in section \ref{convergence}.
In section \ref{application} we point to some particularities when applying the trust-region method to $H_\infty$-optimization.  
This concerns the choice of working model,  trial step, and stability barrier, as well as
the approximation error between the infinite-dimensional $H_\infty$-program and its
discretization.

Numerical results for applications to infinite-dimensional control problems are presented in section \ref{numerics}. 
Subsection \ref{illustration} shows how the method is, in general, applied to a boundary control problem,
subsection
\ref{1D} illustrates this in boundary $H_\infty$-control of a non-linear  reaction-convection-diffusion equation, 
and subsection \ref{reactor} for a non-linear Van de Vusse reactor.
Subsection \ref{sect_cavity} discusses a cavity flow control problem. 

\section{Control strategy}
\label{strategy}
We consider an abstract linear time-invariant control system of the form
\begin{align}
\label{system}
G: \quad  \left\{ 
\begin{array}{ll}
\dot{x} &= Ax + Bu \\
y &= Cx + Du
\end{array}
\right.
\end{align}
where $A$ is an unbounded linear operator on a Hilbert space $Z$  generating a strongly continuous semi-group,
$B$ a closed linear operator mapping the control input space $U$ to $Z$, 
$C$ a closed linear operator mapping $Z$ to the space $Y$ of measured outputs, and $D$ a closed linear operator mapping $U$ to $Y$.
For practical reasons we assume that $U \simeq \mathbb R^p$ and $Y \simeq \mathbb R^m$,   which reflects the fact that the process
is assessed by a finite number of sensors and actuators, rendering
our control law physically implementable.  This means that
$B,D$ are bounded, while $C$  is allowed to be closed unbounded. In addition,  the domains
of $A$ and $C$ satisfy  $D(A) \subset D(C)$.
As a consequence of the finite rank assumption on $Y,U$,
the transfer function $G(s) = C (sI-A)^{-1} B + D$ is defined on the resolvent set $\rho(A)$ and meromorphic on $\mathbb C$,
with values  $G(s)\in \mathbb C^p \times \mathbb C^m$, see e.g. \cite{JN:1988,EN}. 

We consider a class $K\in \mathscr K$ of feedback control laws which have similar state-space realizations
\begin{align}
\label{controller}
K: \quad \left\{ 
\begin{array}{cl}
\dot{x}_K \!& = A_Kx_K + B_Ky \\
u\! &= C_Kx_K + D_Ky
\end{array}
\right.
\end{align}
on a Hilbert space $Z_\mathscr K$, with input space $Y \simeq \mathbb R^m$ and output
space $U \simeq \mathbb R^p$, so that we can put $G$ and $K\in\mathscr K$ in lower feedback $\mathcal F_\ell(G,K)$ as in Figure \ref{figure1}.
Candidate controllers $K\in \mathscr K$ have to stabilize $\mathcal F_\ell(G,K)$  internally
in closed loop, by which we mean that the infinitesimal generator of the closed loop system generates an exponentially
stable semi-group \cite{JN:1988,zwart}. 

\begin{figure}[ht!]
%\begin{pspicture}(2,5)
%\psline(4,3)(4,5)(6,5)(6,3)(4,3)
%\uput[0](4.65,4){$P$}
%\psline[arrows=->,arrowsize=5pt 1](2.5,4.5)(4,4.5)  \uput[0](2.3,4.8){$w$}
%\psline[arrows=->,arrowsize=5pt 1](6,4.5)(7.5,4.5)  \uput[0](7.1,4.8){$z$}
%
%\psline(4.4,2.1)(5.6,2.1)(5.6,1.1)(4.4,1.1)(4.4,2.1)
%\uput[0](4.65,1.6){$K$}
%\psline[arrows=->,arrowsize=5pt 1](4.4,1.6)(3.5,1.6)(3.5,3.5)(4,3.5)
%\psline[arrows=->,arrowsize=5pt 1](6,3.5)(6.5,3.5)(6.5,1.6)(5.6,1.6)
%
%\uput[0](3,2.55){$u$}
%\uput[0](6.4,2.55){$y$}
%\end{pspicture}
\centerline{
\includegraphics[scale=0.32]{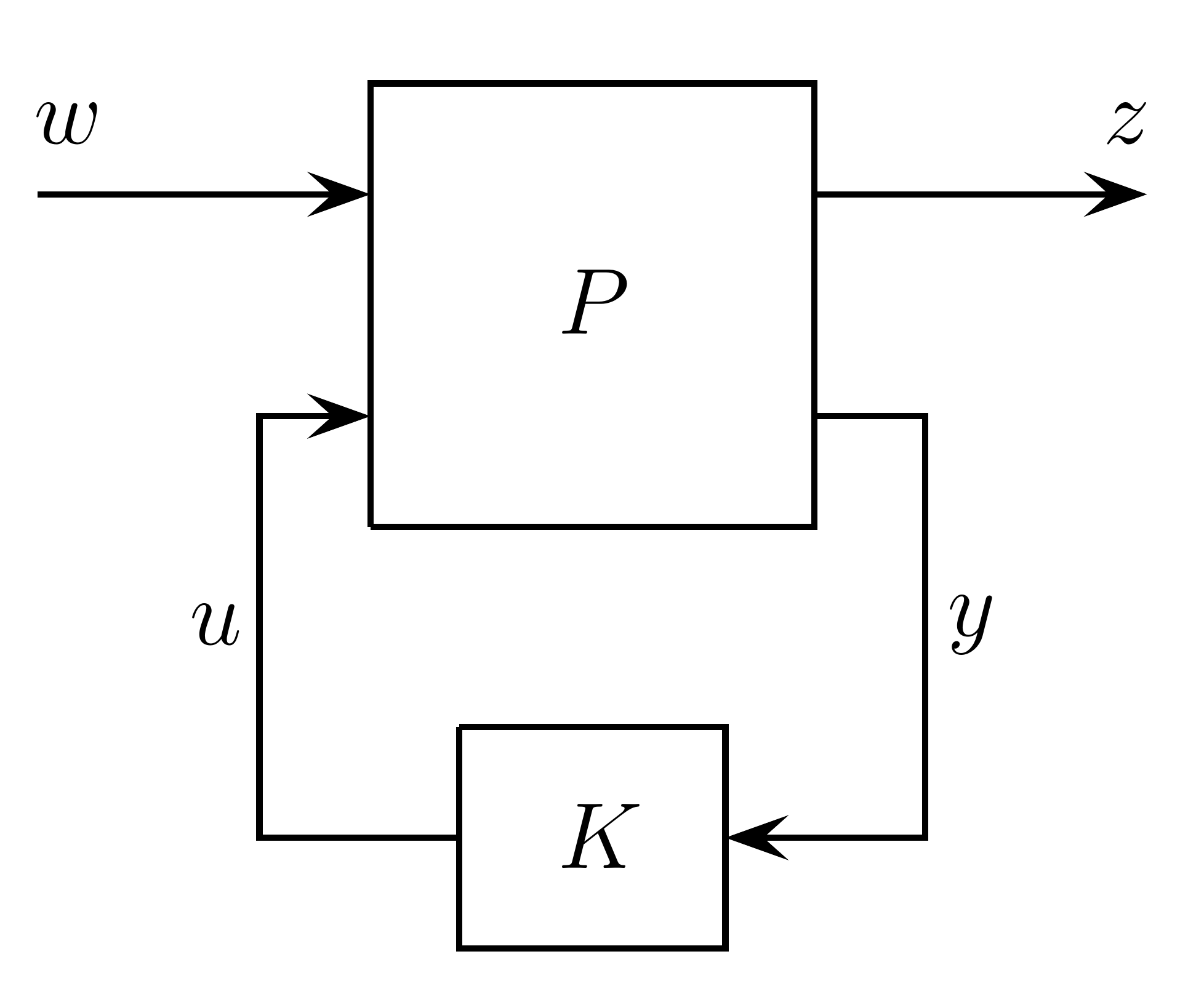}
}
\vspace*{-.1cm}
\caption{\small Lower feedback interconnection $\mathcal F_\ell(P,K)$ with $H_\infty$-performance channel $w \to z$. \label{figure1}}
\end{figure}

In $H_\infty$-control one has not only to assure stability in closed loop, but also to guarantee good performance
and robustness of the feedback system. For that purpose the system $G$ is embedded in a plant
$P$ with a similar state-space realization
\begin{align}
\label{plant}
P: \quad  \left\{
\begin{array}{rl}
\dot{x} &= \,Ax \,+ \,B_1 w\, + \,B_2 u \\
z &= C_1 x + D_{11} w + D_{12} u \\
y &= C_2x + D_{21} w + D_{22} u
\end{array}
\right.
\end{align}
where the channel $y\to u$ is used for control, the channel $w\to z$ for performance. Closing the $u$-$y$-loop  in (\ref{plant}) with (\ref{controller}) as in 
Figure \ref{figure1} leaves us with the closed-loop transfer function
$T_{wz}(K)$ 
from exogenous input $w$ to regulated output $z$.
The $H_\infty$-control problem consists now in minimizing   the $L^2$-$L^2$-operator norm of $T_{wz}(K)$ over a suitable class $K\in \mathscr K$ of admissible control laws.
This operator  norm is also known as
the $H_\infty$-norm, given as
\begin{equation}
\label{boyd}
\|T_{wz}(K) \|_\infty = \max_{\omega \in [0,\infty]} \overline{\sigma} \left( T_{wz}(K,j\omega)\right),
\end{equation}
where $\overline{\sigma}(M)$ denotes the maximum singular value of a matrix $M$. 
If candidate controllers $K\in \mathscr K$ in (\ref{controller}) are parametrized as $K(\x)$ for a finite-dimensional vector $\x\in \mathbb R^n$ of tunable parameters, then  
this $H_\infty$-optimization program takes the form 
\begin{eqnarray}
\label{infinite_hinfstruct}
\begin{array}{ll}
\mbox{minimize} & f_\infty(\x) = \displaystyle\max_{\omega\in [0,\infty]} \overline{\sigma}\left( T_{wz}(K(\x),j\omega) \right)  \\
\mbox{subject to} & \mbox{$K(\x)$ stabilizes $G$ in closed-loop} \\
&\x\in \mathbb R^n
\end{array}
\end{eqnarray}
For finite-dimensional real-rational $P,K(\x)$ the objective
(\ref{boyd}) may be computed by an iterative procedure  \cite{boyd1,boyd2}, but for infinite dimensional $G,K(\x)$, this is no longer possible, and we need to approximate (\ref{boyd})
on a finite grid $\Omega_{\rm opt} \subset [0,\infty]$. Introducing the discretized version
\[
\|T_{wz}(K)\|_{\infty,d} = \max_{\omega\in \Omega_{\rm opt}} \overline{\sigma}\left(  T_{wz}(K,j\omega) \right)
\]
of the $H_\infty$-norm on the grid $\Omega_{\rm opt}$, this
leads to the discretized  $H_\infty$-optimization program
\begin{eqnarray}
\label{hinfstruct}
\begin{array}{ll}
\mbox{minimize} & f(\x) = \displaystyle\max_{\omega\in \Omega_{\rm opt}} \overline{\sigma}\left( T_{wz}(K(\x),j\omega) \right)  \\
\mbox{subject to} & \mbox{$K(\x)$ stabilizes $G$ in closed-loop} \\
&\x\in \mathbb R^n
\end{array}
\end{eqnarray}
to which our non-smooth optimization method is applied.
The overall procedure for $H_\infty$-control
is now given in algorithm \ref{algo3}.

\begin{algorithm}[ht]
\caption{Infinite-dimensional $H_\infty$-synthesis \label{algo3}}
\begin{algo}
\PARAMETERS Tolerance $\vartheta >0$.
\STEP{Steady-state} Compute steady-state and linearize infinite-dimensional system about steady-state.
\STEP{Transfer function}
Use linearized infinite-dimensional system to compute transfer function $G(j\omega)$, either formally for arbitrary $\omega \in [0,\infty]$, or numerically at a very high precision
on a very fine grid $\Omega_{\rm fine}$.
\STEP{Plant} Set up plant $P$ which defines the $H_\infty$-performance channel $w\to z$.
\STEP{Grid for optimization} Find initially stabilizing $K(\x_0)$ for $G$
and use it to compute grid for optimization $\Omega_{\rm opt}$ such that
$ \|T_{wz}(K(\x_0))\|_\infty \leq  f(\x_0)  + \vartheta$. Either use a formal or a numerical function $G(j\omega)$, or extract $\Omega_{\rm opt}$ from
the pre-computed high precision grid $\Omega_{\rm fine}$.
\STEP{Non-smooth optimization} Use non-smooth trust-region algorithm \ref{algo1} to compute locally optimal solution
$K(\x^*)$ of  (\ref{hinfstruct}).
%\STEP{Stability} Use Nyquist criterion to test {\em a posteriori}  whether $K(\x^*)$ is closed-loop stabilizing. 
\STEP{Refined grid} Check whether $\| T_{wz}(K(\x^*))\|_\infty \leq \| T_{wz}(K(\x^*))\|_{\infty,d} + \vartheta$. If not then add nodes to $\Omega_{\rm opt}$
and go back to step 5.
\end{algo}
\end{algorithm}

\begin{remark}
It should be stressed that  the computation of $G(j\omega)$ in steps 1 and 2 of algorithm \ref{algo3} is the only moment where the full infinite-dimensional,
or likewise, large-scale finite-dimensional  model is used. Since this step is performed prior to optimization, and typically
$|\Omega_{\rm opt}| \ll | \Omega_{\rm fine}|$,  optimization is speedy.  Also, the process
of finding the correct $P$, which needs going back to step 3 of algorithm \ref{algo3}, is not slowed down by steps 1 and 2. In other words,
once the transfer function $G(j\omega)$ is available, the complexity of the process is the same as
that of a finite-dimensional structured $H_\infty$-design procedure in the sense of \cite{tac,barratt}.
This will be illustrated in our experimental section.
\end{remark}

\begin{remark}
Stability in closed-loop  in step 4, and during optimization in step 5, uses the well-known Nyquist stability test in tandem with the barrier
approach to be addressed in section \ref{barrier}. For the theoretical justification in the context of infinite-dimensional systems see \cite{IJRNC}.
\end{remark}

We now address the individual steps of algorithm \ref{algo3}. The central ingredient is optimization, 
which is needed to solve (\ref{hinfstruct}), and
which is discussed in the next sections \ref{optimize} and \ref{convergence}. 
Generation of the grid $\Omega_{\rm opt}$ in step 4, and the  certificate in step 6,  are discussed in section \ref{application}.

\section{Non-smooth trust-region method}
\label{optimize}
We work with the non-smooth trust-region method introduced
in \cite{ANR:2016}, which has already been successfully used
in mechanical contact problems \cite{gwinner}, and in system theory \cite{ANR:2016} for computing the worst-case
$H_\infty$-performance of a system, its stability margin, and its distance to instability.  Here we use it for 
$H_\infty$-control of infinite-dimensional systems.  We also answer a question left open in \cite[Remark 16]{ANR:2016}, which
concerns the theoretical justification of Kiwiel's aggregation technique \cite{kiwiel} in non-smooth trust-regions. This question goes back to  \cite{rus}
for the convex trust-region method, but had until now remained  open. An affirmative 
answer will be obtained in section \ref{convergence}.

\subsection{Presentation of the algorithm}
We briefly recall the essentials of the non-smooth trust-region method. For the present work it is sufficient to apply 
it to optimization programs of the form
\begin{equation}
\label{program}
\min_{\x \in \mathbb R^n} f(\x),
\end{equation}
where $f:\mathbb R^n \to \mathbb R$ is locally Lipschitz but non-smooth and non-convex. Following \cite{noll},  a function
$\phi:\mathbb R^n \times \mathbb R^n \to \mathbb R$ is called a {\em model} of $f$ if it satisfies the following
properties:
\begin{itemize}
\item[$(M_1)$] $\phi(\cdot,\x)$ is convex, $\phi(\x,\x) = f(\x)$, and $\partial_1\phi(\x,\x) \subset \partial f(\x)$.
\item[$(M_2)$] If $\y_k\to \x$, then there exist $\epsilon_k\to 0^+$ such that $f(\y_k) \leq \phi(\y_k,\x)+ \epsilon_k\|\y_k-\x\|$.
\item[$(M_3)$] If $\x_k\to \x, \y_k\to \y$, then $\limsup_{k\to\infty} \phi(\y_k,\x_k) \leq \phi(\y,\x)$.
\end{itemize}
We may interpret $\phi(\cdot,\x)$ as a substitute for the first-order Taylor expansion of  $f$ at $\x$.  For convergence
theory we need a slightly stronger type of model, which is given  by the following:

\begin{definition}
A first-order model $\phi$ of $f$ is called {\em strict} if it satisfies the following stronger version of axiom $(M_2)$:
\begin{itemize}
\item[$(\widehat{M}_2)$] If $\x_k,\y_k\to \x$, then there exist $\epsilon_k\to 0^+$ such that $f(\y_k) \leq \phi(\y_k,\x_k)+ \epsilon_k\|\y_k-\x_k\|$.
\end{itemize}
\hfill $\square$
\end{definition}

\noindent
The difference between $(M_2)$ and the strict version $(\widehat{M}_2)$ is analogous to the difference between 
differentiability and strict differentiability, hence the nomenclature.
For additional information on the model concept see
\cite{noll,ANP:2009,ANR:2016}.

\begin{remark}
\label{max_eig}
A typical example of a strict model $\phi$ is obtained when $f$ is a maximum eigenvalue function
$f(\x) = \lambda_1 \left( F(\x)\right)$, with $F:\mathbb R^n \to \mathbb S^m$ a class $C^1$-mapping into the space of $m\times m$ Hermitian matrices.
We take
$\phi(\y,\x) = \lambda_1 \left( F(\x) + F'(\x)(\y-\x) \right)$. See \cite{noll,ANP:2009}.
\end{remark}

\begin{definition}
{\rm
Let $\x$ be the current serious iterate of the trust-region algorithm, $\z$ a trial step. Let $g$ be a subgradient of $\phi(\cdot,\x)$ at $\z$.
Then the affine function $m_\z(\cdot,\x)=\phi(\z,\x) + g^\top (\cdot - \z)$ is called a} cutting plane {\rm of $f$ at serious iterate $\x$
and trial step $\z$.} \hfill $\square$
\end{definition}

If $\z=\x$, then due to axiom $(M_1)$ a cutting plane $m_\x(\cdot,\x)$ at serious iterate $\x$ and trial step $\z=\x$ is just a tangent plane to $f$ at $\x$. 
Since $m_\x(\x,\x)=f(\x)$, including $m_\x(\cdot,\x)$ in the working model $\phi_k(\cdot,\x)$ at $\x$ guarantees exactness
$\phi_k(\x,\x)=f(\x)$ of the working model at all counters $k$.
We cast this in the following

\begin{definition}
A cutting plane $m_\x(\cdot,\x)$ at serious iterate $\x$ and trial step $\z=\x$  is called  an {\em exactness
plane}.  \hfill $\square$
\end{definition}

As is standard in bundle and cutting plane algorithms, by storing cutting planes at unsuccessful trial steps $\z^k$, we accumulate information, which
we use to build polyhedral models of $f$ near $\x$. We use the notation
$\phi_k(\cdot,\x)$ for these {\em working models} of $f$ formed by cutting planes, where $k$ denotes the counter of the inner loop.
Note that $\phi_k\leq \phi$ by construction of the cutting planes. 
If in addition a positive semi-definite symmetric matrix $Q(\x)\succeq 0$ is available as a substitute for the Hessian of $f$ at $\x$, then
we call $\Phi_k(\cdot,\x) = \phi_k(\cdot,\x) + \frac{1}{2} (\cdot-\x)^\top Q(\x) (\cdot-\x)$
a {\em second-order working model} of $f$ at serious iterate $\x$.

We are now ready to present the bundle trust-region algorithm. (See algorithm \ref{algo1}).

\begin{algorithm}[ht!]
\caption{Non-smooth trust-region algorithm \label{algo1}}
\begin{algo}
\PARAMETERS $0 < \gamma < \widetilde{\gamma} < 1$, $0 < \gamma < \Gamma \leq 1$, $0 < \theta \ll 1$, $M \geq 1$, $q > 0$.
\STEP{Initialize outer loop} Fix initial iterate $\x^1$ and memory trust-region radius $R_1^\sharp > 0$. Initialize $Q_1 \succeq 0$ with $\|Q_1\|\leq q$.
Put outer loop counter $j=1$.
\STEP{Stopping test} At outer loop counter $j$, stop if $\x^j$ is a critical point of (\ref{program}). Otherwise go to inner loop.
\STEP{Initialize inner loop}  Put inner loop counter $k=1$ and initialize trust-region radius as $R_1 = R_j^\sharp$. Build polyhedral first-order working model 
$\phi_1(\cdot,\x^j)$, where at least one exactness plane at $\x^j$ is included. Possibly enrich by adding recycled planes from previous steps, or by including
anticipated cutting planes. Build second-order working model
$\Phi_1(\cdot,\x^j) = \phi_1(\cdot,\x^j) + \frac{1}{2}(\cdot-\x^j)^\top Q_j (\cdot-\x^j)$.
\STEP{Trial step generation} At inner loop counter $k$ compute solution $\y^k$ of trust-region tangent program \\
\vspace*{.1cm}
\hspace*{4cm}$\mbox{minimize} \quad\; \Phi_k(\y,\x^j)$ \\
\hspace*{4cm}$\mbox{subject to}  \;\; \; \|\y-\x^j\| \leq R_k$\\
\vspace*{.1cm}
Then admit any  $\z^k$ satisfying $\|\z^k-\x^j\| \leq M \| \y^k-\x^j\|$ and $f(\x^j)-\Phi_k(\z^k,\x^j) \geq
\theta \left( f(\x^j) - \Phi_k(\y^k,\x^j) \right)$ as trial step. 
\STEP{Acceptance test} If
%\hspace*{3cm}
$$\rho_k =\displaystyle \frac{f(\x^j)-f(\z^k)}{f(\x^j)-\Phi_k(\z^k,\x^j)} \geq \gamma$$ 
put $\x^{j+1}=\z^k$ (serious step), quit inner loop and goto step 8. Otherwise (null step), continue inner loop with step 6.
\STEP{Update working model}
Generate a cutting plane $m_k(\cdot,\x^j)$ of $f$ at the unsuccessful trial step $\z^k$ and add it to the polyhedral model. Possibly taper out
$\phi_k$ by removing some of the older cuts, and build new first-order working $\phi_{k+1}(\cdot,\x^j)$. Then
$\Phi_{k+1}(\cdot,\x^j)=\phi_{k+1}(\cdot,\x^j) + \frac{1}{2}(\cdot-\x^j)^\top Q_j (\cdot-\x^j)$ is the new second-order working model. Continue with step 7.
\STEP{Update trust-region radius} Compute secondary control parameter \\
\hspace*{3cm}$$\widetilde{\rho}_k =\displaystyle \frac{f(\x^j)-\phi_{k+1}(\z^k,\x^j)}{f(\x^j)-\Phi_{k}(\z^k,\x^j)}   $$
and put 
$$R_{k+1}=\bigg\{ {R_k \; \, \,\,\mbox{ if }\; \widetilde{\rho}_k < \widetilde{\gamma} \atop \frac{1}{2}R_k \; \mbox{ if } \;\widetilde{\rho}_k \geq \widetilde{\gamma}}$$
Increase inner loop counter $k$ and go back to step 4.
\STEP{Update memory radius} Store new memory trust-region radius
$$ R_{j+1}^\sharp = \bigg\{ { R_k\, \,\,\; \mbox{ if }\; \rho_k < \Gamma \atop 2 R_k \; \mbox{ if }\; \rho_k \geq \Gamma}$$
Update $Q_j \to Q_{j+1}$ respecting $Q_{j+1} \succeq 0$ and $\|Q_{j+1}\| \leq q$. Increase outer loop counter $j$ and go back to step 2.
\end{algo}
\end{algorithm}

\newpage
\begin{remark}
Before we discuss convergence of algorithm \ref{algo1} in the next section, we recall the form of the tangent program in step 4 from \cite{ANR:2016}. 
Let the first-order working model
at inner loop counter $k$ have the form $\phi_k(\cdot,\x^j) = \max_{i\in I_k} a_i + g_i^\top (\cdot - \x^j)$ for some finite set $I_k$, and suppose the
trust-region norm is the maximum norm. Then the tangent  program  at serious iterate $\x^j$ and inner loop instant $k$ is the following CQP
\begin{eqnarray}
\label{tangent}
\begin{array}{ll}
\mbox{minimize} & t + \frac{1}{2} (\y-\x^j)^\top Q_j(\y-\x^j) \\
\mbox{subject to} & a_i + g_i^\top (\y-\x^j) \leq t, \;\; i\in I_k\\
&-R_k \leq \y_i - \x_i^j \leq R_k \;\;, i=1,\dots,n
\end{array}
\end{eqnarray}
with decision variable $(t,\y) \in \mathbb R\times \mathbb R^n$, giving rise to the solution $\y^k$ in step 4.
\end{remark}

\subsection{Convergence}
\label{convergence}
In this section we prove  convergence of the trust-region  algorithm \ref{algo1} toward a Clarke critical point. As is standard,
we start by proving that
the inner loop ends finitely if $0\not\in \partial f(\x^j)$, where throughout $\partial f$ denotes the Clarke subdifferential.
During this part of the proof we write $\x = \x^j$ and $Q = Q_j$, as those are fixed during the inner loop at 
counter $j$. Note that by the necessary optimality condition
for the tangent program in step 4 of algorithm \ref{algo1}, there exists a subgradient 
$g_k^*\in \partial_1 \phi_k(\y^k,\x)$ such that $g_k^*+Q(\y^k-\x)+\vv_k=0$, where $\vv_k$ is in the normal cone
to the trust-region norm ball $B(\x,R_k)$ at $\y^k$. 

\begin{definition}
We call $g_k^*$ the {\em aggregate subgradient}.
The affine function
$m_k^*(\cdot,\x) = \phi_k(\y^k,\x) + g_k^{*\top} (\cdot - \y^k)$ is called the {\em aggregate plane}. \hfill $\square$
\end{definition}

\begin{lemma}
\label{est}
There exists  $\sigma > 0$ depending only on the constants
$\theta \in (0,1)$ and $M>0$ in algorithm {\rm \ref{algo1}} and on the trust-region norm $\|\cdot\|$, such that for every trial point $\z^k$ at inner loop instant $k$ with corresponding
solution $\y^k$ of the trust-region tangent program in step {\rm 4}, and for the corresponding aggregate subgradient $g^*_k \in \partial_1 \phi(\y^k,\x)$, we have the estimate
\begin{equation}
\label{kolmogoroff}
f(\x)-\phi_k(\z^k,\x) \geq \sigma \|g^*_k + Q(\y^k-\x)\| \| \z^k-\x\|.
\end{equation}
\end{lemma}

\begin{proof}
This is essentially the same as \cite[Lemma 1]{ANR:2016}. \hfill $\square$
\end{proof}

\begin{lemma}
Suppose the inner loop at $\x$ turns infinitely, and $\liminf_{k\to \infty} R_k=0$. Then $\x$ is a critical point of {\rm (\ref{program})}.
\end{lemma}

\begin{proof}
According to step 7 of  algorithm \ref{algo1}
we have $\widetilde{\rho}_k \geq \widetilde{\gamma}$ for infinitely many $k\in \mathcal K$. Since $R_k$ is never increased during the inner loop, that implies
$R_k\to 0$. Hence $\y^k,\z^k\to \x$ as $k\to \infty$, where we use the trial step generation rule of step 4 of algorithm \ref{algo1}. We argue that this implies $\phi_k(\z^k,\x) \to f(\x)$. 

Indeed,
$\limsup_{k\to\infty} \phi_k(\z^k,\x) \leq \limsup_{k\to \infty} \phi(\z^k,\x) = \lim_{k\to\infty} \phi(\z^k,\x)=f(\x)$ is always true due to
$\phi_k\leq \phi$ and axiom $(M_1)$. On the other hand, $\phi_k$ includes (i.e. dominates) an exactness plane $m_0(\cdot,\x)=f(\x) + g_0^\top(\cdot-\x)$, hence  
$f(\x) = \lim_{k\to\infty} m_0(\z^k,\x) \leq \liminf \phi_k(\z^k,\x)$. The two together show $\phi_k(\z^k,\x)\to f(\x)$, and then immediately also 
$\Phi_k(\z^k,\x)\to f(\x)$. We also readily obtain $\phi_k(\y^k,\x) \to f(\x)$ from the link between $\z^k,\y^k$ in step 4 of algorithm \ref{algo1}.

We now prove that $\liminf_{k\to \infty}\|g_k^*\|= 0$.
Assume on the contrary that $\|g_k^*\|\geq \eta > 0$ for all $k$. Choose $k$ large enough to have
$\|g_k^*+Q(\y^k-\x)\| \geq \frac{1}{2} \|g_k^*\|$. 
{\color{black} This is possible because the sequence $g_k^*$ is bounded away from 0 and  $Q({\bf y}^k-{\bf x})\to 0$.}
Then by  estimate  (\ref{kolmogoroff}) we have
$f(\x)-\phi_k(\z^k,\x) \geq \frac{1}{2} \sigma\eta \| \z^k-\x\|$.  Therefore, for $k$ large enough,
$f(\x)-\Phi_k(\z^k,\x) \geq \frac{1}{4}\sigma\eta \| \z^k-\x\|$, as the quadratic term in $\Phi_k$
is of the order $\| \z^k-\x\|^2$.
Since $\z^k\to \x$, by axiom $(\widehat{M}_2)$ there exist
$\epsilon_k\to 0^+$ such that
$f(\z^k) - \phi(\z^k,\x) \leq \epsilon_k \| \z^k-\x \|$. 
%But then also $f(\z^k) - \Phi(\z^k,\x) \leq \widetilde{\epsilon}_k \| \z^k-\x \|$ with $\widetilde{\epsilon}_k = \epsilon_k + \frac{1}{2}(\z^k-\x)^\top Q (\z^k-\x) \to 0$. 
Now we estimate
\[
\widetilde{\rho}_k = \rho_k + \frac{f(\z^k)- \phi(\z^k,\x)}{f(\x)-\Phi_k(\ \z^k,\x)}
\leq \rho_k + \frac{{\epsilon}_k \| \z^k-\x\| }{ \frac{1}{4} \sigma \eta \| \z^k-\x \|} = \rho_k+ 4{\epsilon}_k/(\sigma\eta).
\]
Since ${\epsilon}_k \to 0$ and $\rho_k < \gamma$, we have $\limsup \widetilde{\rho}_k \leq \gamma < \widetilde{\gamma}$,
a contradiction with $\widetilde{\rho}_k > \widetilde{\gamma}$ for the infinitely many $k\in \mathcal K$. That proves $g_k^*\to 0$
for a subsequence $k\in \mathcal N$. 

Next observe that by the subgradient inequality and $\phi_k \leq \phi$ we have
\[
g_k^{*\top} \h \leq \phi_k(\y^k+ \h,\x) - \phi_k(\y^k,\x) \leq \phi(\y^k+\h,\x)-\phi_k(\y^k,\x).
\]
Since $\phi_k(\y^k,\x) \to f(\x)=\phi(\x,\x)$, passing to the limit $k\in \mathcal N$ and using $g_k^*\to 0$, $\y^k\to \x$
implies
\[
0 \leq \phi(\x+\h,\x) - \phi(\x,\x).
\]
Since $\h$ was arbitrary, we have $0\in \partial_1 \phi(\x,\x) \subset \partial f(\x)$ by $(M_1)$. That proves the Lemma.
\hfill $\square$
\end{proof}

This result needs only the fact that $\phi_k \leq \phi$, so it is not in conflict with
any rule in step 6 used to taper out $\phi_k \to \phi_{k+1}$, as long as some exactness plane is present in $\phi_k$ at all inner loop instants $k$. 
In the next  two lemmas we examine the more involved case when $R_k$
is bounded away from 0. Here we require not only that an exactness plane is present at all times, but also
that the latest cutting plane is added into $\phi_{k+1}$.

However, this leads to a tangent program of size growing with $k$,
which raises the question whether it is in principle possible to limit the number  of planes included in the working model $\phi_k(\cdot,\x)$ in step 6. For the convex bundle
method this question is answered in the affirmative by Kiwiel's aggregation rule 
\cite{kiwiel},  according to which only three planes are required, an exactness plane, the
latest cut, and the aggregate plane to account for the past. For the non-convex bundle method, an affirmative answer was first given in \cite{noll}.
The aggregate plane is a convex combination of active cuts at the trial point $\y^k$, and can be described as follows. Had we removed from the
last tangent program all active planes, and substituted instead the aggregate plane, the solution $\y^k$ would have been the same.  
The idea of the aggregation technique is to add the aggregate plane into the working model after an unsuccessful trial step $\y^k$,  which allows to remove
the active planes for the next sweep. Inactive planes may leave the model in any case.

\begin{remark}
{\color{black} Ruszczy\'nski  \cite{rus} stresses that the situation is more delicate for the convex non-smooth trust region method, and asks whether
convergence could be proved under the aggregation rule, or under any other rule allowing to limit the number of
planes included in the $\phi_k$.  Here we address this question in the general non-convex case. In \cite{ANR:2016} we had shown that the number of
cuts in step 6 may at least be limited to $n+2$ using Carath\'eodory's theorem, but we remarked that it would be far more attractive to have a maximum number
independent of the dimension $n$ as in Kiwiel's rule.  In 
\cite[Remark 16]{ANR:2016}
we observed that the question whether Kiwiel's rule could also be justified for the trust-region method was still open. 
Here we shall answer this question in the affirmative by proving convergence under the aggregation rule.} 
\end{remark}

The following result justifies the use of aggregation in the first place for the special case $\z^k=\y^k$.  Note that 
the trivial choice $\z^k=\y^k$ in step 4 is always authorized (due to $M \geq 1$ and $\theta \leq 1$), but
of course we want to use the additional freedom offered by $\z^k$ to improve performance of our method, so $\z^k=\y^k$
is rather restrictive, and we will remove it later.

\begin{lemma}
\label{long}
Suppose the inner loop at $\x$ turns infinitely and the trust-region radius $R_k$ stays bounded away from $0$.
Let $Q \succ 0$ and suppose the aggregation rule is used to taper out the models in step {\rm 6}. Suppose the $\y^k$ are chosen as trial steps.
Then $\x$ is
a critical point of {\rm (\ref{program})}.
\end{lemma}

\begin{proof}
Since the trust-region radius is frozen $R_k=R_{k_0}$ from some counter $k_0$ onwards, we write $R := R_{k_0}$.
According to step 7 of the algorithm that means $\widetilde{\rho}_k < \widetilde{\gamma}$ for $k\geq k_0$. 
The only progress in the working model
as we update $\phi_k \to \phi_{k+1}$
is now the addition of the cutting plane and the aggregate plane.  {\color{black} The working models $\phi_k$ now contain at least three planes,
an exactness plane, the latest cut from the last unsuccessful trial step, and the aggregate plane. They may  contain more planes, 
but those will not be used in our argument below.}

We want to prove $\y^k\to \x$. Since $R_k$ stays bounded away
from 0, this is more involved than in the previous Lemma.
Since $Q \succ 0$ is fixed, we introduce the Euclidian norm $|\x|_Q^2 = \x^\top Q\x$. With this arrangement the objective function
of the tangent program becomes
\[
\Phi_k(\cdot,\x) = \phi_k(\cdot,\x) + \textstyle\frac{1}{2} | \cdot - \x|_Q^2.
\] 
We know that the cutting plane $m_k(\cdot,\x)$ at  trial step $\y^k$ satisfies $m_k(\y^k,\x) = \phi(\y^k,\x)$,
so it memorizes the value $\phi(\y^k,\x)$, while the aggregate plane $m_k^*(\cdot,\x)$ satisfies
$m_k^*(\y^k,\x) = \phi_k(\y^k,\x)$, so it memorizes in turn the value $\phi_k(\y^k,\x)$. The latter gives
\begin{equation}
\label{what}
\Phi_k(\y^k,\x) = m_k^*(\y^k,\x) + \textstyle \frac{1}{2} |\y^k-\x|_Q^2.
\end{equation}
Now we introduce the quadratic function
\[
\Phi^*_k(\cdot,\x) = m_k^*(\cdot,\x) + \textstyle \frac{1}{2} |\cdot - \x|_Q^2,
\]
then from what we have just seen in (\ref{what})
\begin{equation}
\label{equal}
\Phi_k^*(\y^k,\x) = \Phi_k(\y^k,\x).
\end{equation}
Moreover, we have
\begin{equation}
\label{monotone}
\Phi_k^*(\cdot,\x) \leq \Phi_{k+1}(\cdot,\x),
\end{equation}
because according to the aggregation rule we include the aggregate plane $m_k^*(\cdot,\x)$ in the built of the new model
$\phi_{k+1}$, that is, we have $m_k^*(\cdot,\x) \leq \phi_{k+1}(\cdot,\x)$, and hence (\ref{monotone}).
Expanding the quadratic function $\Phi_k^*(\cdot,\x)$ at $\y^k$ gives
\[
\Phi_k^*(\cdot,\x) = \Phi_k^*(\y^k,\x) + \nabla \Phi_k^*(\y^k,\x)^\top (\cdot-\y^k) +\textstyle \frac{1}{2}| \cdot-\y^k|_Q^2,
\]
where $\nabla \Phi_k^*=g_k^* + Q(\y^k-\x)$. 
From the optimality condition of the tangent program at $\y^k$ we get  $g_k^*+Q(\y^k-\x) = - \vv_k$ with $\vv_k$ in the normal cone to the 
ball $B(\x,R)$ at $\y^k$, hence 
\begin{equation}
\label{just}
\Phi_k^*(\cdot,\x) = \Phi_k^*(\y^k,\x) - \vv_k^\top(\cdot-\y^k) + \textstyle \frac{1}{2} |\cdot - \y^k|_Q^2.
\end{equation}
Now we argue as follows:
\begin{align}
\label{chain}
\begin{split}
\Phi_k(\y^k,\x) &= \Phi_k^*(\y^k,\x)   \hspace*{4cm}   \mbox{ (by (\ref{equal})) } \\
&\leq \Phi_k^*(\y^k,\x) + \textstyle\frac{1}{2} |\y^{k+1}-\y^k|_Q^2 \\
&= \Phi_k^*(\y^{k+1},\x) + \vv_k^\top (\y^{k+1}-\y^k)  \quad \mbox{ (by (\ref{just}))} \\
&\leq \Phi_{k}^*(\y^{k+1},\x)   \hspace*{3.6cm}   \mbox{(since $\vv_k^\top(\y^{k+1}-\y^k) \leq 0$) } \\
&\leq \Phi_{k+1}(\y^{k+1},\x)   \hspace*{3.25cm}   \mbox{(by (\ref{monotone}))} \\
&\leq \Phi_{k+1}(\x,\x)  \qquad\qquad\qquad\qquad\quad\,   \mbox{($\y^{k+1}$ minimizer of $\Phi_{k+1}(\cdot,\x)$)} \\
&= \phi(\x,\x) = f(\x).
\end{split}
\end{align}
Therefore the sequence $\Phi_k(\y^k,\x)$ is increasing and bounded above, and converges to a limit $\Phi^* \leq f(\x)$. Going back with this information to the estimation chain (\ref{chain})
shows $\frac{1}{2}|\y^{k+1}-\y^k|_Q^2 \to 0$ and also $\vv_k^\top (\y^{k+1}-\y^k) \to 0$. Then also
\[
\textstyle \frac{1}{2} |\y^{k+1}-\x|_Q^2 - \frac{1}{2} |\y^k-\x|_Q^2 \to 0,
\]
because $|\cdot|_Q$ is a Euclidian norm. In consequence
\[
\phi_{k+1}(\y^{k+1},\x) - \phi_k(\y^k,\x) = \Phi_{k+1}(\y^{k+1},\x) - \Phi_k(\y^k,\x) -\textstyle \frac{1}{2}|\y^{k+1}-\x|_Q^2+\frac{1}{2} |\y^k-\x|_Q^2 \to 0.
\]
Now recall  that the cutting plane $m_k(\cdot,\x)$ is an affine support function of $\phi_{k+1}(\cdot,\x)$ at $\y^{k}$.
Hence by the subgradient inequality
\[
g_k^\top (\cdot - \y^k) \leq \phi_{k+1}(\cdot,\x) - \phi_{k+1}(\y^k,\x).
\]
Since $\phi_{k+1}(\y^k,\x) = \phi(\y^k,\x)$, we deduce
\begin{equation}
\label{previous}
\phi(\y^k,\x) + g_k^\top(\cdot-\y^k)\leq \phi_{k+1}(\cdot,\x).
\end{equation}
Now using (\ref{previous}) we estimate as follows:
\begin{align*}
0 &\leq \phi(\y^k,\x) - \phi_k(\y^k,\x) \\
&= \phi(\y^k,\x) + g_k^\top(\y^{k+1}-\y^k) - \phi_k(\y^k,\x) - g_k^\top(\y^{k+1}-\y^k)\\
&\leq \phi_{k+1}(\y^{k+1},\x) - \phi_k(\y^k,\x) - g_k^\top(\y^{k+1}-\y^k).
\end{align*}
Since $\y^{k+1}-\y^k\to 0$ and the $g_k$ are bounded, we have $g_k^\top(\y^{k+1}-\y^k)\to 0$, hence
we deduce $\phi(\y^{k},\x)-\phi_k(\y^k,\x) \to 0$, and also $\Phi(\y^k,\x) - \Phi_k(\y^k,\x) \to 0$.

Now we claim that  $\phi_k(\y^k,\x)\to f(\x)$. Since
$\phi_k(\y^k,\x)\leq \Phi_k(\y^k,\x) \to \Phi^* \leq f(\x)$, it remains to prove $\liminf \phi_k(\y^k,\x) \geq f(\x)$.
Suppose that this is not the case, and let $\phi_k(\y^k,\x) \to f(\x) - \eta$ for a subsequence and some $\eta > 0$.
Then also $\phi(\y^k,\x) \to f(\x)-\eta$ for that subsequence. {\color{black} (Here we use that
$\phi_k(\y^k,\x)-\phi(\y^k,\x)\to 0$ proved above).} Passing to yet another subsequence,  and using boundedness of the $\y^k$, we may assume
$\frac{1}{2}|\y^k-\x|_Q^2 \to \ell \geq 0$.
Choose $\delta > 0$ such that $\delta < (1-\widetilde{\gamma})\eta$. From what we have just seen there exists $k_1$ such that
\[
\phi(\y^k,\x) - \phi_k(\y^k,\x) < \delta
\]
for all $k \geq k_1$. Now recall that $\widetilde{\rho}_k \leq \widetilde{\gamma}$ for every $k\geq k_0$, hence
\[
\widetilde{\gamma} \left(  \Phi_k(\y^k,\x) - f(\x) \right)
\leq \phi(\y^k,\x) - f(\x)
\leq \phi_k(\y^k,\x)-f(\x) + \delta 
. 
\]
Passing to the limit 
gives $-\widetilde{\gamma} \eta + \widetilde{\gamma} \ell \leq -\eta +  \delta$, hence $(1-\widetilde{\gamma})\eta + \ell \widetilde{\gamma} \leq \delta$,
which contradicts the choice of $\delta$.
Hence $\phi_k(\y^k,\x) \to f(\x)$.  We immediately deduce that $\Phi_k(\y^k,\x)\to f(\x)$ and $\Phi(\y^k,\x)\to f(\x)$.

We now argue that $\y^k\to \x$. This follows from the estimates
\[
\phi_k(\y^k,\x) \leq
\Phi_k(\y^k,\x) =\phi_k(\y^k,\x) + \textstyle\frac{1}{2}|\y^k-\x|_Q^2 \leq \Phi^* \leq f(\x)
\]
because $\phi_k(\y^k,\x)\to f(\x)$ now shows that all terms go to $f(\x)$, and that implies $|\y^k-\x|_Q \to 0$.
Since $Q \succ 0$ we deduce $\y^k\to \x$, and this is where the proof no longer works if only $Q \succeq 0$. Note that this shows that $\y^k$ is in the interior of $B(\x,R)$
from some counter onward, so that $\vv_k=0$.

Let us now show that $0\in \partial f(\x)$. From the subgradient inequality we have
\[
g_k^{*\top}(\x+\h - \y^k) \leq \phi_k(\x+\h,\x) - \phi_k(\y^k,\x) \leq \phi(\x+\h,\x) - \phi_k(\y^k,\x).
\]
Passing to the limit using $g_k^*\to 0$, $\phi_k(\y^k,\x) \to f(\x) = \phi(\x,\x)$, we obtain
\[
0 \leq \phi(\x+\h,\x) - \phi(\x,\x),
\]
and since $\h$ is arbitrary and $\phi(\cdot,\x)$ is convex, this gives $0 \in \partial_1 \phi(\x,\x) \subset \partial f(\x).$ 
\hfill $\square$
\end{proof}

\begin{remark}
For the general case $Q \succeq 0$ and $\z^k$
different from $\y^k$ it is still not known whether aggregation is justified, but
with \cite[Lemma 2]{ANR:2016} we can prove convergence if we keep all
cuts in the model, or if we use the Carath\'eodory type argument of that reference
to limit the number of cutting planes in the inner loop to $\leq n+2$. 
\end{remark}

\begin{remark}
Note that $Q \succ 0$ is not a restriction  in practice, but $\z^k=\y^k$ is. Fortunately the aggregation technique may still be justified in the case
$\z^k\not=\y^k$
if we proceed as follows. %We use $Q\succ 0$,  which is beneficial anyway, as it has a stabilizing effect on the solution $\y^k$ of the tangent program. and in 
In the first place we
allow $\z^k$ as a trial point in  step 4. If acceptance fails,
then we perform step 7. However, if step 7 gives no reduction of $R_k$, then we are in the difficult case.
We then do the following. We fall back on $\y^k$
as the trial point, i.e., we forget about $\z^k$. 
When $\y^k$ is not accepted, we proceed with step 6 and apply aggregation.  This is now justified
because we are  in the situation covered by Lemma \ref{long}. Note that the additional work required in steps 6 and 7 is marginal, so we do not 
waste time by this  evasive maneuver.  We could even perform this maneuver as default (i.e. checking $\y^k$ if $\z^k$ fails). 
We have proved the following
\end{remark}

\begin{lemma}
Suppose the inner loop turns infinitely.
Suppose $Q \succ 0$ and that the aggregation rule is used in step {\rm 6} to limit the size of $\phi_k$ to any pre-defined fixed number $N \geq 3$.
Suppose we accept to fall back on $\y^k$ if $\z^k$ fails in step {\rm 5} with $\widetilde{\rho}_k < \widetilde{\gamma}$ in step {\rm 7}.
 Then $0\in \partial f(\x)$.
 \hfill $\square$
\end{lemma}

\begin{remark}
In the classical trust-region method failure of the trial step always
leads to reduction of the trust-region radius. One occasionally sees non-smooth versions where authors
do the same. As we have already shown in \cite[section 5.5]{ANR:2016}, that must fail. The example in that reference also shows that
the Cauchy point fails in the non-smooth case. 
\end{remark}

We are now ready to state the main convergence result for our optimization method. The  proof may be adapted from \cite{ANR:2016} with minor changes,
so we skip it here.

\begin{theorem}
Suppose $\x^1$ is such that the level set $\{\x\in \mathbb R^n: f(\x) \leq f(\x^1)\}$ is bounded.  Let $\x^j$ be the sequence of serious
iterates generated by the bundle trust-region algorithm. Then every accumulation point $\x^*$ of the $\x^j$ is a critical point of 
{\rm (\ref{program})}.
\hfill $\square$
\end{theorem}

\section{The case of $H_\infty$-synthesis}
\label{application}
The general purpose
algorithm \ref{algo1} is readily applicable to the $H_\infty$-design problem (\ref{hinfstruct}), as this is 
a special case of (\ref{program}).  Since $f$ is the square-root of a maximum eigenvalue function,
the ideal model $\phi$ in algorithm \ref{algo1} is chosen as in remark \ref{max_eig}. 
There are, however,  some particularities in the application of algorithm \ref{algo1} to (\ref{hinfstruct}),  on which we comment in this section.

\subsection{Stability barrier}
\label{barrier}
The hidden constraint of closed-loop stability may occasionally 
lead to unacceptable trial points $\y^k$, but this can be avoided by complementing the objective
in (\ref{hinfstruct}) by the following barrier function:
$f(\x) = \max\{\|T_{wz}(K(\x))\|_{\infty,d},c\|S(K(\x))\|_{\infty,d}\}$, where $S = (I+GK)^{-1}$ is the closed-loop sensitivity function,
and $c>0$ a small constant. It is well-known that $\|S\|_\infty^{-1}$, also known as the modulus margin,  is an indicator for the distance of the Nyquist curve
to the point  $-1$, and since the Nyquist curve crosses $-1$ when iterates become destabilizing, 
the term $\|S\|_\infty$ becomes large as iterates approach the limit of the region of stability. In other words,
$c\|S\|_\infty$ has the effect of a barrier function at the boundary of the hidden constraint. Note that the maximum
of two $H_\infty$-norms is again a $H_\infty$-norm, i.e., we may still represent the modified
objective as $f(\x) = \|T_{w',z'}(K)\|_{\infty,d}$ for a modified channel $w' \to z'$.  Note that other closed-loop transfer function, which
may contribute a similar complementary  barrier effect include $GS = G(I+GK)^{-1}$, $(I-KG)^{-1}$ and $K(I+KG)^{-1}$, see e.g. \cite[p.36]{barratt}.

\subsection{Exploiting freedom in steps 3 and 4}
An important practical aspect of program (\ref{hinfstruct}) is to use an adapted initialization
of the model $\phi_1(\cdot,\x)$ at the beginning of the inner loop, the idea being that in the vast majority of cases the first trial step will then be successful.
This is achieved by including not only active frequencies $\omega$ from  $T_{wz}(K(\x),j\omega)$ at $\x$ in the model, but also branches belonging to secondary peaks,
which are susceptible to become active at the next trial step.  Selecting near active frequencies of
the $H_\infty$-norm is decisive for the quality of the working models $\phi_k$ and was already discussed in \cite{tac} and \cite{ANP:2009} in the context of the bundle method,
and this is shown schematically in Figure \ref{fig-vartheta}. We refer to this type of affine functions as anticipated cutting
planes, and their integration into the working models is covered by convergence theory as long as they are
affine minorants of $\phi$.

\begin{figure}[ht!]
\begin{center}
\includegraphics[scale=0.35]{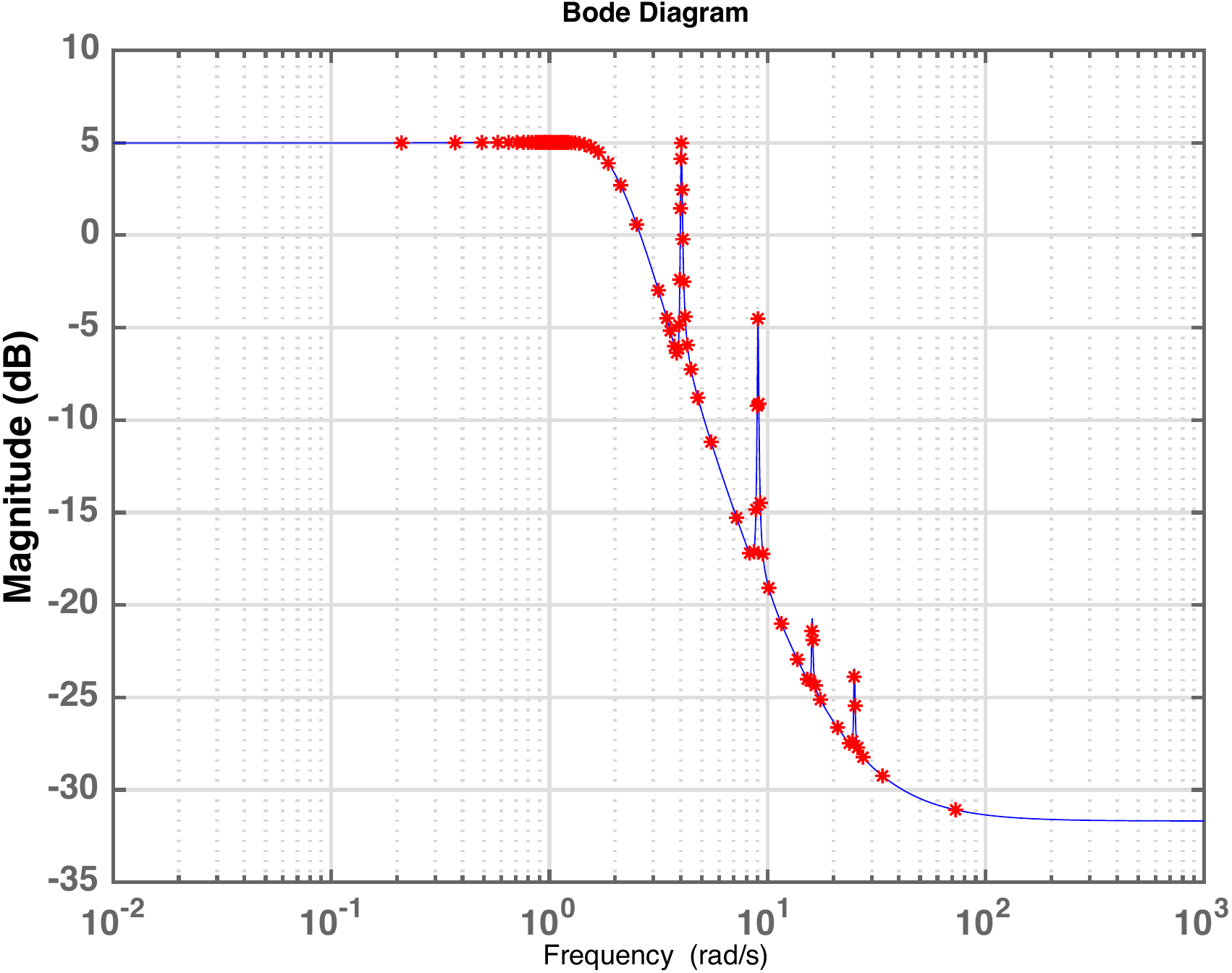}
$\qquad$\includegraphics[scale=0.7]{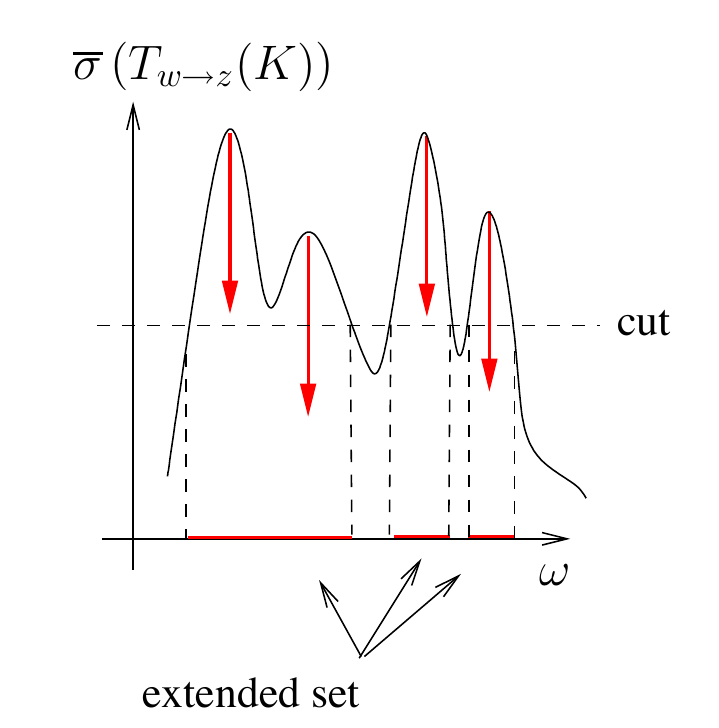}
\end{center}
\caption{\small Left: grid $\Omega_{\rm opt}$ based on the criterion (\ref{var}). Right:  selection of extended set of frequencies around active
frequencies with primary and secondary peaks (i.e., global and local maxima of $\omega\mapsto \overline{\sigma}(T_{wz}(j\omega)$).  \label{fig-vartheta}}
\end{figure}

\begin{remark}
We next exploit the freedom in step
4 of algorithm \ref{algo1} in the case of program (\ref{hinfstruct}).  In case of failure of the solution $\y^k$ of the tangent program
it is attractive to use a backtracking linesearch to  generate trial steps of the form $\z^k = \x + t(\y^k-\x)$, $0 < t < 1$. 
By convexity, $\Phi_k(\x,\x) - \Phi_k(\x+t(\y^k-\x),\x)
\geq t \left(  \Phi_k(\x,\x) - \Phi_k(\y^k,\x)\right)$, hence every $\z^k$ with $t\geq \theta$ gives automatically a trial point in the sense of step 4, 
and we only have to check acceptance $\rho_k\geq \gamma$.
Even for smaller steps
$t < \theta$ it may still be possible to have  $f(\x^j)-\Phi_k(\z^k,\x^j) \geq
\theta \left( f(\x^j) - \Phi_k(\y^k,\x^j) \right)$, in which case $\z^k$ remains a candidate. 
\end{remark}

\subsection{Performance certificate}
\label{certificate}
The strategy in algorithm \ref{algo3} is to perform discretization
at the level of the system transfer function $G(s)$, and not before, avoiding system reduction or identification.
To justify this we have to select a discretization $\Omega_{\rm opt}$ for optimization (\ref{hinfstruct})
such that the optimal value $f(\x^*)=\|T_{wz}(K(\x^*))\|_{\infty,d}$ is not too far from the true infinite-dimensional value
$f_\infty(\x^*)=\|T_{wz}(K(\x^*))\|_\infty$. This hinges on a suitable grid generation technique. This is crucial in our approach, but we stress
that this takes place in a low-dimensional space, whereas system reduction and identification need heavy large scale linear algebra 
machinery, and yet remain on the level of heuristics.

Given a continuously differentiable function $\phi:[0,\infty] \to \mathbb R_+$, we    want to construct a finite grid
$\Omega_{\rm opt}$ such that $\max_{\omega \in [0,\infty]} | \phi(\omega) - P_\phi(\omega) | \leq \vartheta$ for a fixed tolerance $\vartheta$, where $P_\phi$
is the piecewise linear function interpolating $\{\phi(\omega): \omega \in \Omega_{\rm opt}\}$. 
We call $M[\cdot,\cdot]$ a first-order bound of $\phi: \mathbb R \to \mathbb R$ if
$|\phi'(\omega)| \leq M[\omega^-,\omega^+]$ for all $\omega^- < \omega^+$ and all $\omega \in [\omega^-,\omega^+]$.
We need the following preparatory:

\begin{lemma}
\label{katz}
Suppose we have constructed a grid $\Omega$ on $[0,\infty]$ such that for two consecutive nodes
$\omega_i,\omega_{i+1}\in \Omega$
and some $\gamma^* \geq \max\{\phi(\omega_i),\phi(\omega_{i+1})  \}$  the inequality
\begin{equation}
\label{var}
M[\omega_i,\omega_{i+1}] (\omega_{i+1}-\omega_i) < 2 \gamma^* + 2 \vartheta - \phi(\omega_i) - \phi(\omega_{i+1}).
\end{equation}
is satisfied. 
Then $\phi(\omega) < \gamma^* + \vartheta$ for every $\omega\in [\omega_i,\omega_{i+1}]$.
\end{lemma}

\begin{proof}
Suppose on the contrary that there exists $\omega^* \in [\omega_i,\omega_{i+1}]$ such that $\phi(\omega^*) \geq \gamma^* + \vartheta$. Then the polygon
through $\phi(\omega_i)$, $\phi(\omega^*)$, $\phi(\omega_{i+1})$ has length greater than or equal to $L$, where
\[
L = \sqrt{A^2 + (\omega^*-\omega_i)^2} + \sqrt{B^2+(\omega_{i+1}-\omega^*)^2}
\]
with $A = \gamma^*+\vartheta - \phi(\omega_i)$ and $B=\gamma^*+ \vartheta - \phi(\omega_{i+1})$. Now $L \geq \ell$, where
$$
\ell := \min_{\omega \in [\omega_i,\omega_{i+1}]}  \sqrt{A^2 + (\omega-\omega_i)^2} + \sqrt{B^2+(\omega_{i+1}-\omega)^2}
= \sqrt{(A+B)^2 + (\omega_{i+1}-\omega_i)^2},
$$ 
the minimum being attained at $\omega = \frac{\omega_i B + \omega_{i+1}A}{A+B}$.
On the other hand the
curve $\{(\omega,\phi(\omega)): \omega \in [\omega_i,\omega_{i+1}]\}$ has length
$\mathscr L = \int_{\omega_i}^{\omega_{i+1}}\sqrt{1+\phi'(\omega)^2}d\omega \leq \sqrt{1+M[\omega_i,\omega_{i+1}]^2} (\omega_{i+1}-\omega_i)$, and
$\mathscr L \geq L$, so in combining the two estimates we get $\mathscr L \geq \ell$, which yields the estimate
$\sqrt{1+M[\omega_i,\omega_{i+1}]^2} \geq \sqrt{{(A+B)^2}/{(\omega_{i+1}-\omega_i)^2}+1}$.   We deduce $M[\omega_i,\omega_{i+1}] \geq (A+B)/(\omega_{i+1}-\omega_i)$,
and since $A+B=2\gamma^*+2\vartheta - \phi(\omega_i)-\phi(\omega_{i+1})$, this contradicts (\ref{var}).
\hfill $\square$
\end{proof}

%We consider a twice differentiable function $\phi:[0,\infty] \to \mathbb C$.   We  want to construct a finite grid
%$\Omega_{\rm opt}$ such that $\max_{\omega \in [0,\infty]} | \phi(\omega) - P_\phi(\omega) | \leq \vartheta$, where $P_\phi$
%is the linearly interpolating polygon of $\{\phi(\omega): \omega \in \Omega_{\rm opt}\}$. We call $M[\cdot,\cdot]$ a second-order bound of $\phi$, if
%$|\phi''(\omega)| \leq M[\omega_1,\omega_2]$ for all $0 \leq \omega_1 < \omega_2 \leq \infty$ and $\omega\in [\omega_1,\omega_2]$. 

%We have the following

%\begin{lemma}
%\label{katz}
%Suppose we have constructed a grid $\Omega$ on $[0,\infty]$ such that for all $\omega_i\in \Omega$
%\begin{equation}
%\label{vartheta}
%M[\omega_i,\omega_{i+1}](\omega_{i+1}-\omega_i)^2 < 8 \vartheta.
%\end{equation}
%Then $|\phi(\omega)-P_\phi(\omega)| < \vartheta$ for all $\omega\in [0,\infty]$.
%\end{lemma}

%\begin{proof}
%By the Lagrange interpolation formula we have
%$\phi(\omega)-P_\phi(\omega)= \phi''(\tilde{\omega}) (\omega-\omega_i)(\omega-\omega_{i+1})/2$ for every $\omega\in [\omega_i,\omega_{i+1}]$ and some
%$\tilde{\omega} \in [\omega_i,\omega_{i+1}]$ depending on $\omega$. Therefore 
%$\displaystyle\max_{\omega\in [\omega_i,\omega_{i+1}]}|\phi(\omega) - P_\phi(\omega) |\leq \max_{\omega\in [\omega_i,\omega_{i+1}]} |\phi''(\omega)/2|
%|(\omega-\omega_i)(\omega-\omega_{i+1})| \leq M[\omega_i,\omega_{i+1}]/2 (\omega_{i+1}-\omega_i)^2/4 < \vartheta$.
%\hfill $\square$
%\end{proof}

We would now like to apply this to the function
$\phi(\omega) = \overline{\sigma}\left( T_{wz}(K^*,j\omega)\right)$, where $K^*=K(\x^*)$ is the optimal $H_\infty$-controller
computed by algorithm \ref{algo1}. For that we have to prove differentiability of $\phi$. We have the following:

\begin{lemma}
\label{boyd}
{\rm \cite[Theorem 2.3]{boyd2}}  The function
$\phi$ has only a finite number of points of non-smoothness, and in particular, is of class $C^2$ in the neighborhood of all
primary and secondary peaks  (all global and local maxima).
\end{lemma}

\begin{proof}
By \cite[Thm. 6.1]{kato} the one-parameter family
of Hermitian matrices $\omega \mapsto \mathscr T(\omega) = T_{wz}(K^*,j\omega)^HT_{wz}(K^*,j\omega)$ has real analytic eigenvalue functions $\lambda_\nu(\omega)$,
hence $\phi^2(\omega)$ is a finite maximum of real analytic functions, and then also $\phi$
because $\phi > 0$. The rest of the argument is now as in \cite{boyd2}.
\hfill $\square$
\end{proof}

\begin{remark}
In consequence, $\phi$ is twice continuously differentiable in a neighborhood of each peak,
and in particular on a set $\{\omega\in [0,\infty]: \phi(\omega) > \|T_{wz}(K^*)\|_\infty-\vartheta_0\}$ for some $\vartheta_0 > 0$. This means  Lemma \ref{katz} is applicable.
\end{remark}

We use this to construct the grid $\Omega_{\rm opt}$ used in (\ref{hinfstruct}) as follows.
Start with $\omega_0=0$. Having constructed $\omega_i$, compute an extrapolation
$\omega_i^\sharp> \omega_i$ and obtain $M = \max \{M[\omega_i,\omega]: \omega \in [\omega_i,\omega_i^\sharp]\}$. Then choose
$\omega_{i+1} \in (\omega_i,\omega_i^\sharp]$ such that 
\begin{equation}
\label{rule}
M (\omega_{i+1} - \omega_i) < 2\max\{\phi(\omega_i),\phi(\omega_{i+1})\}+2 \vartheta-\phi(\omega_i)-\phi(\omega_{i+1}).
\end{equation}
 %Then loop on.
If $G(s)$ is available analytically, then $M[\cdot,\cdot]$ is computable. In the numerical approach we use a finite difference
estimation $\phi'(\omega) \approx (\phi(\omega^+)-\phi(\omega)) /(\omega^+-\omega)$. Since $\phi$ is continuously differentiable near the peak values, this gives excellent results.  
In our experience, 
the method rarely leads to grids with more than a few hundred of nodes, which allows an efficient solution of the optimization program. 
A typical example is shown in Figure \ref{fig-vartheta}.

The following result  justifies our method theoretically.

\begin{theorem}
If $0 < \vartheta \leq \vartheta_0$ and if a first-order bound $M[\cdot,\cdot]$ for $\phi = \overline{\sigma}\left( T_{wz}(K^*,\cdot) \right)$ in tandem with rule
{\rm (\ref{rule})} is used in step {\rm 4} of
algorithm {\rm \ref{algo3}} to construct $\Omega_{\rm opt}$, then the gain $\gamma^*$ achieved by the solution $K^*$ of 
{\rm (\ref{hinfstruct})} is {\rm certified} to satisfy $\gamma^* \geq \|T_{wz}(K^*)\|_\infty - \vartheta$.
%\hfill $\square$
\end{theorem}

\begin{proof}
Since $\phi(\omega) = \overline{\sigma}\left(  T_{wz}(K^*,j\omega) \right)$ and $\|T_{wz}(K^*)\|_\infty = \max_{\omega \in [0,\infty]} \phi(\omega)$, 
we have to show that $\gamma^* \geq \phi(\omega) - \vartheta$ for every $i$ and all $\omega \in [\omega_i,\omega_{i+1}]$, where $\omega_i$ are the nodes
of the grid $\Omega_{\rm opt}$  constructed in step 4 of algorithm \ref{algo3} based on the rule (\ref{rule}). Since $\gamma^*$ is the gain achieved by the solution
$K^*$ of (\ref{hinfstruct}) on that grid, it satisfies $\gamma^* =\max_i \phi(\omega_i)$. Hence $\gamma^* \geq \max\{\phi(\omega_i),\phi(\omega_{i+1})\}$, 
and so condition (\ref{var}) of Lemma \ref{katz} is satisfied. Since by Lemma \ref{boyd} we may apply Lemma \ref{katz} to $\phi$, 
we obtain the conclusion $\phi(\omega) \leq \max\{\phi(\omega_i),\phi(\omega_{i+1})\}+\vartheta \leq  \gamma^*+\vartheta$ on $[\omega_i,\omega_{i+1}]$, as claimed.
\hfill $\square$
\end{proof}

\begin{remark}
In  rule (\ref{rule}) we apply (\ref{var}) with $\gamma^* = \max \{ \phi(\omega_i),\phi(\omega_{i+1})\}$ on each interval $[\omega_i,\omega_{i+1}]$.
When it comes to just certifying the optimal value
$f(\x^*)=\|T_{wz}(K(\x^*))\|_\infty = \phi(\omega^*)$ in step 6 of algorithm \ref{algo3}, then  we can construct an even coarser grid by applying (\ref{var}) with $\gamma^*=\phi(\omega^*)$
the same for all $[\omega_i,\omega_{i+1}]$. Namely,  our grid can be very coarse at frequencies $\omega$ where
$\phi(\omega) \ll \phi(\omega^*)$, and still capture sharp peaks, as illustrated in Figure \ref{fig-vartheta}.   We refer to this as a verification grid $\Omega_{\rm ver}$.
According to our experience, 
the outlined method  to construct $\Omega_{\rm opt}$ is   well-adapted to discretize
the controller design problem.
\end{remark}

We can further exploit lemma \ref{katz} to obtain information on how close the values $\gamma^*$ of (\ref{hinfstruct}) 
and $\gamma_\infty$ of the infinite-dimensional program (\ref{infinite_hinfstruct}) are.
 Writing as before
$f(\x) = \|T_{wz}(K(\x))\|_{\infty,d}$ for the discrete $H_\infty$-norm on $\Omega_{\rm opt}$,  and 
$f_\infty(\x) = \|T_{wz}(K(\x))\|_\infty$ for the true $H_\infty$-norm, we compare the discretized
$H_\infty$-program $\min_\x f(\x)$, i.e., (\ref{hinfstruct}), to the underlying infinite-dimensional $\min_\x f_\infty(\x)$, i.e. (\ref{infinite_hinfstruct}). 

\begin{corollary}
Let $\x_\infty$ be a local minimum of {\rm (\ref{infinite_hinfstruct})} with value $\gamma_\infty$,
and $\x^*$ a local minimum of {\rm (\ref{hinfstruct})} with value $\gamma^*$. 
Suppose a first-order bound in tandem with rule {\rm (\ref{rule})} has been used in step {\rm 6} of algorithm {\rm \ref{algo3}}. 
Then if $\x^*$, $\x_\infty$ are within neighborhoods of local optimality of each other, we have
$f(\x_\infty) \geq f(\x^*) \geq f_\infty(\x^*)-\vartheta \geq f_\infty(\x_\infty)-\vartheta \geq f(\x_\infty)-\vartheta$.
\end{corollary}

\begin{proof}
Indeed, $f(\x_\infty) \geq f(\x^*)$ because $\x^*$ is a minimum of $f$ on a neighborhood $U(\x^*)$, and $\x_\infty \in U(\x^*)$
by hypothesis. Next $f(\x^*) \geq f_\infty(\x^*)-\vartheta$ by Lemma \ref{katz}, because construction of the grid uses the bound $M[\cdot,\cdot]$ and rule (\ref{var}).
Next $f_\infty(\x^*) \geq f_\infty(\x_\infty)$, because $\x_\infty$ is a minimum of $f_\infty$ on a neighborhood
$U(\x_\infty)$, and $\x^*\in U(\x_\infty)$ by hypothesis. The last inequality is satisfied because $f \leq f_\infty$. 
\hfill $\square$
\end{proof}

This means comparable locally optimal values of the infinite dimensional $H_\infty$-program (\ref{infinite_hinfstruct})  and its approximation (\ref{hinfstruct})
differ by at most $\vartheta$, our apriori chosen tolerance.
Since most of the time our algorithm finds the global minimum of (\ref{hinfstruct}), this is a very useful result in practice, as
it determines the value of the infinite dimensional $H_\infty$-program (\ref{infinite_hinfstruct}) within a prior tolerance level $\vartheta$.

The argument remains valid if $\x_\infty$, $\x^*$ are only approximate local minima, say up to the same tolerance $\vartheta$ in the values.
Then we get the chain
$f(\x_\infty) \geq f(\x^*)-\vartheta \geq f_\infty(\x^*)-2\vartheta \geq f_\infty(\x_\infty) - 3 \vartheta \geq    f(\x_\infty)-3\vartheta$, so here
our approximation (\ref{hinfstruct}) gives the correct value up to an error of $3\vartheta$ in the values.

\section{Applications}
\label{numerics}
In this section, we apply our method to several challenging studies in control of infinite dimensional system, and in particular,
to boundary and distributed control of systems of parabolic partial differential equations.

\subsection{Computation of $G(s)$  in boundary control}
\label{illustration}
We illustrate how  our procedure is applied to boundary control
of parabolic PDEs. Consider a boundary problem of the form
\begin{align}
\label{parabolic}
\begin{split}
z_t(x,t) = \sum_{|\alpha|,|\beta| \leq m} (-1)^{|\alpha|} D^\alpha (a_{\alpha\beta}(x) D^\beta z)(x,t) = 0& \qquad \mbox{ $(x,t)\in Q \times [0,\infty)$} \\
D^{i-1}_\nu z(x,t) = \mathcal U_i(x,t)  \qquad x \in \partial Q, \; & i=1,\dots,m,
\end{split}
\end{align}
where $\mathcal U_i$ are abstract controls acting on the boundary $\partial Q$.
Here for convenience $Q$ is bounded open with $\partial Q$  a compact orientable $C^\infty$-manifold, the coefficients are $a_{\alpha,\beta}\in C^\infty(\overline{Q})$, and
uniform ellipticity $\sum_{|\alpha|,|\beta| \leq m} a_{\alpha \beta}(x) \xi^\alpha \xi^\beta \geq c |\xi|^2$ is assumed for $x\in Q$.  Then by \cite{salamon} 
problem (\ref{parabolic})  may be  represented in the abstract form (\ref{system}), where however the input space $U$
is potentially  still infinite dimensional. To comply with our assumption that $K(s)$ should be finite-rank, 
i.e., that input and output spaces $U,Y$ should be finite-dimensional,  we
select basis functions $\phi_{ik}$ on the boundary $\partial Q$ and  replace the boundary control action $\mathcal U$ in (\ref{parabolic}) by a finite-dimensional
version
\[
\hspace*{2.6cm}D^{i-1}_\nu z(x,t) = \sum_{k=1}^N \phi_{ki}(x) u_{ik}(t), \quad x\in \partial Q, \; i=1,\dots,m,   %\hspace*{3cm} (\ref{dirichlet})'
\]
which now has input space $U \simeq \mathbb R^{Nm}$. A finite-dimensional output space $Y \simeq \mathbb R^p$ could be obtained by taking
measurements of the form
\[
y_i(t) = \int_Q \psi_i(x) z(x,t) \, dx, \quad i=1,\dots,p,
\]
with another set of basis functions $\psi_i$ on $\Omega$ representing sensors. For one-dimensional $Q$ point
evaluations on $\partial Q$ are possible. 
This case will be used in our numerical experiments.

Referring to \cite{salamon,zwart,sano} for the correct setup of (\ref{system}), we directly pass to the computation of
$G(s)$. Laplace transforming (\ref{parabolic}) with initial condition $z(x,0)=0$ leads for fixed $s\in \mathbb C$ to the elliptic boundary value
problem
\begin{align}
\label{elliptic}
\begin{split}
sz(x,s) = \sum_{|\alpha|,|\beta| \leq m} (-1)^{|\alpha|} D^\alpha (a_{\alpha\beta}(x) D^\beta z)(x,s) = 0 \qquad \mbox{$x \in Q$} \\
D^{i-1}_\nu z(x,s) =  \sum_{k=1}^N \phi_{ki}(x) u_{ik}(t) \qquad x \in \partial Q,   \; i=1,\dots,m. 
\end{split}
\end{align}
Then $G_{ikr}(s)$ is obtained by solving (\ref{elliptic}) with $u_{ik}=1$
$u_{i',k'}=0$ for $(i',k')\not=(i,k)$, and by computing $y_r(s)=\int_Q \psi_r(x) z(x,j\omega) \, dx$.
For the one-dimensional case  $Q=[0,1]$,  point evaluations $y_r(s) = z_r(0,s)$, $y_r(s)= z_r(1,s)$, or linear combinations
of those, are possible. 

\begin{remark}
Computation
of $G'(s)$ can also be obtained by solving an elliptic boundary value problem, which is (\ref{elliptic}) differentiated with respect to $s$.
Since $K$ is known explicitly, this is useful when computing the bound $M[\cdot,\cdot]$ of $\phi$ in Lemma \ref{katz}.
\end{remark}

\begin{remark}
In those cases where computations are not performed formally, a high spatial resolution
is used to solve (\ref{elliptic}) accurately. One such solve  can then be interpreted as a function
evaluation  $G_{ikr}(s)$. If carried out numerically, we perform the computation of
$G(s)$ for fixed $s=j\omega$ with the highest spatial discretization available in our setup. Typically, this is at least as accurate spatially as
our final simulation of the closed-loop system. Since
the number of inputs and outputs is not very large, pre-computing $G(s)$  will not seriously burden the overall performance of
algorithm \ref{algo3}.  Since pre-computing $G(j\omega)$ is done off-line, it neither impedes  the optimization phase, nor
the plant modeling phase.
\end{remark}

\subsection{Reaction-convection-diffusion equation}
\label{1D}
We consider a non-linear reaction-convection-diffusion equation with Danckwaerts and von Neumann boundary conditions
\begin{align}
\begin{split}
\frac{\partial C(z,t)}{\partial t} = D\frac{\partial^2C(z,t)}{\partial z^2} - U(t)\frac{\partial C(z,t)}{\partial z} - kC(z,t)&\qquad (z,t) \in [0,L] \times [0,\infty) \\
D\frac{\partial C(0,t)}{\partial z} - U(t)(C(0,t) - C_{\rm in}) = 0,& \qquad \frac{\partial C(L,t)}{\partial z}=0. 
\end{split}
\end{align}
The process represents a chemical reaction in a cylindrical plug flow reactor with time-varying flow velocity $U(t)$, constant axial
dispersion $D$, and constant reaction rate $k$. The dynamics of the  reaction  $A \stackrel{k}{\to} B$ are described by the spatially and temporally varying concentration $C(z,t)$
of reactant $A$, the concentration of product $B$ being a dependent state. Using online measurement $y(t)=C(L,t)-C_{ss}(L)$  of the concentration of ingredient $A$ at the outflow position $z=L$
we steer the plug flow velocity $U(t)$ to maintain the process in steady-state $U_{ss}$, $C_{ss}(z)$, $y_{ss}$,  while attenuating measurement noise and a disturbance of the flow velocity,
and to enable speedy tracking of set-point changes in the steady-state concentration. {\color{black}We refer to \cite{sano} or \cite[Example 3.3.5]{zwart} for the correct
setup of this problem as a Hilbert space linear system (\ref{system}).}

%\vspace*{-4cm}
  \begin{figure}[ht!]
  \begin{center}
  \includegraphics[scale=0.55]{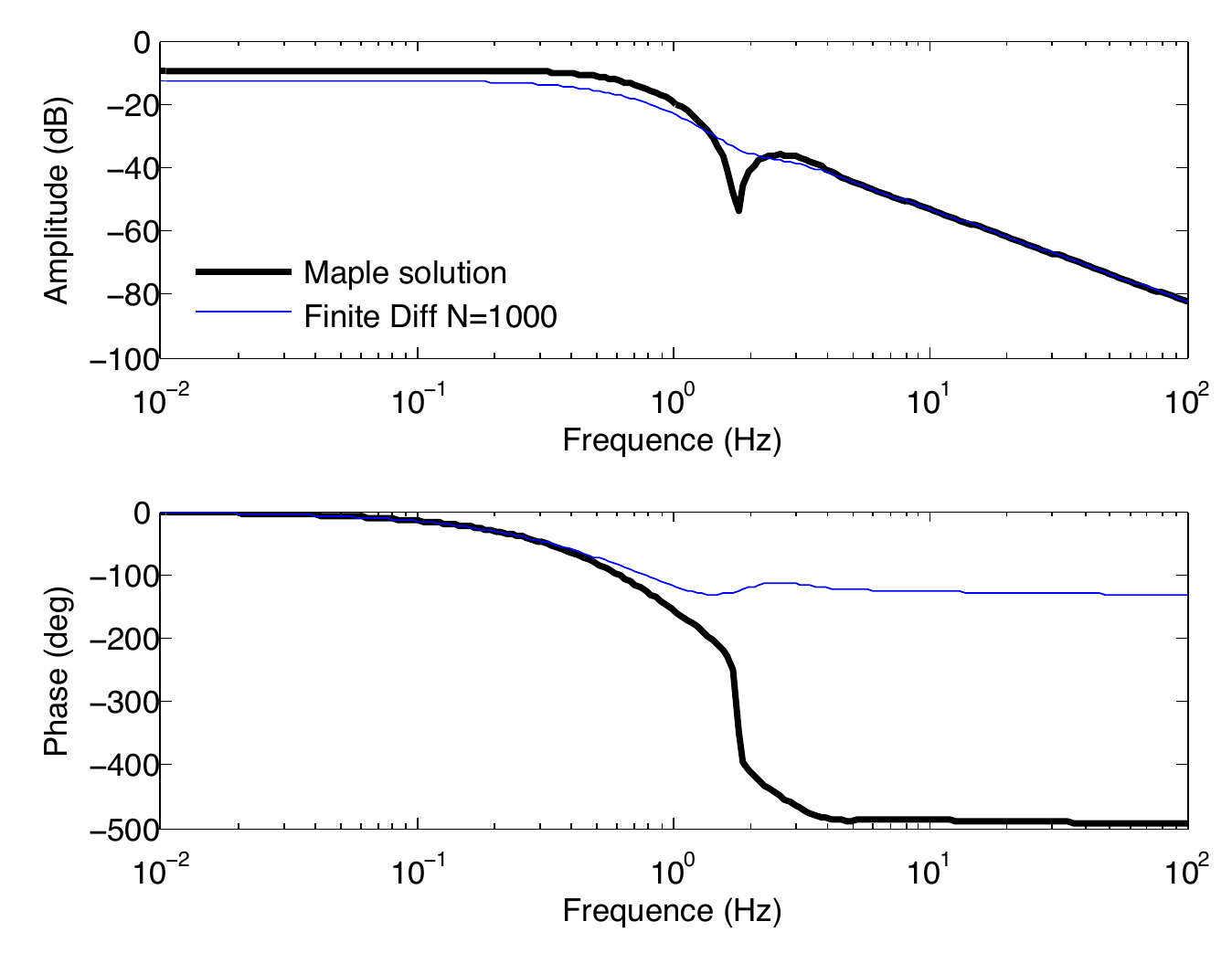}
  \includegraphics[scale=0.55]{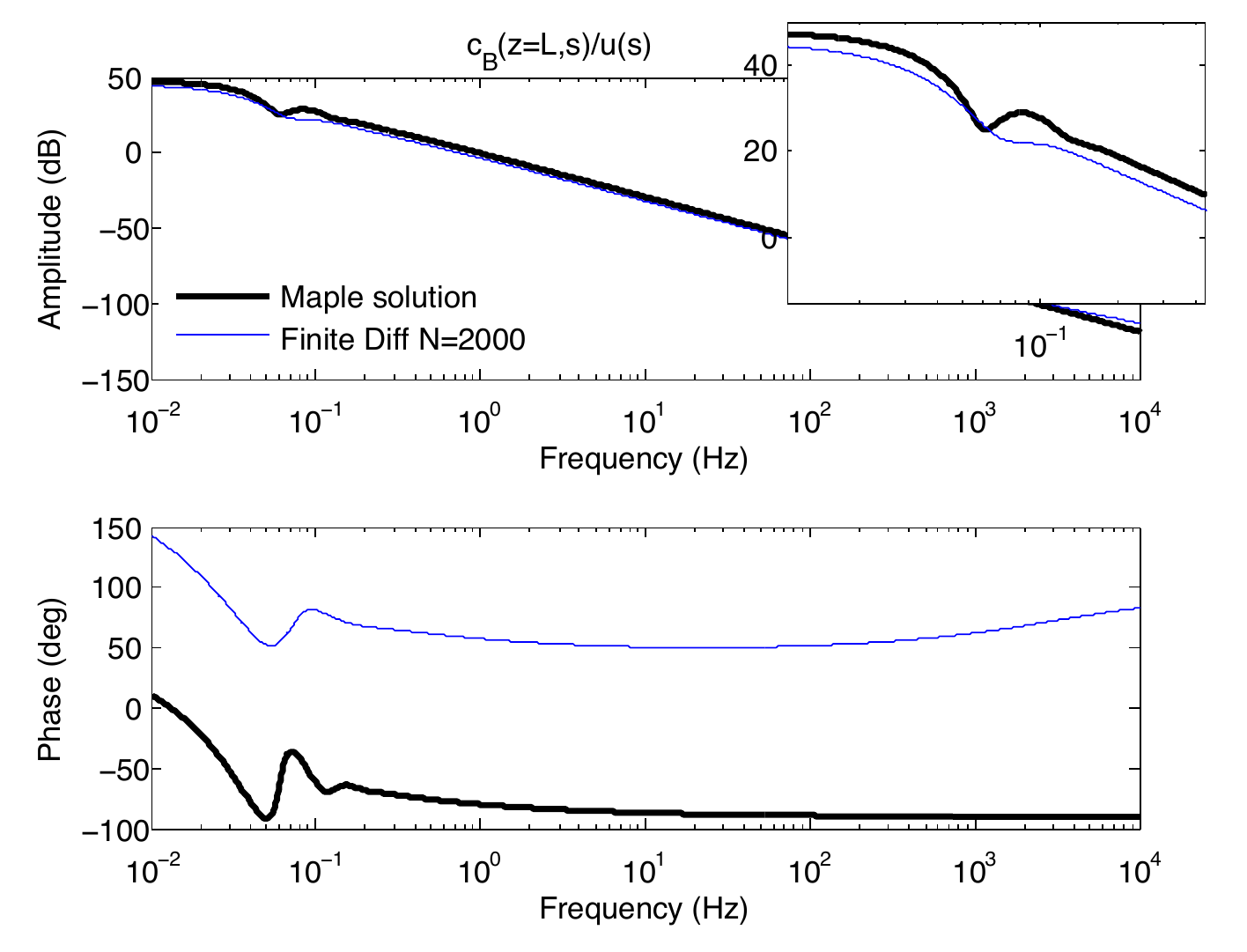}
  \end{center}
  \vspace*{-.4cm}
  \caption{\small Bode plot of infinite-dimensional transfer function $G(s)$ compared with finite-difference based  $G_{fd}(s)$.
  Left shows study 1 computed with Maple, compared to finite-differences of order $N=1000$. Even with 5000 states
 the transmission zero at frequency $\omega 1.6$Hz, which is missed by the discretization of order $N=1000$.  
Right shows Van de Vusse
study with $G(s)$ computed via numerical Maple solve of (\ref{linearized2}), compared to finite-differences of order $N=2000$. \label{compare}}
    \end{figure}

Fixing a steady-state flow velocity $U_{ss}$,  we  
compute the corresponding steady-state concentration $C_{ss}(z)$ by solving the one-dimensional boundary value problem
\begin{align}
\begin{split}
 DC_{ss}''(z) - U_{ss}C_{ss}'(z) - kC_{ss}(z) &=0\\
 DC_{ss}'(0) - U_{ss}(C_{ss}(0) - C_{\rm in}) = 0,& \qquad C_{ss}'(L)=0. 
  \end{split}
\end{align}  
  Linearization about steady-state with $U(t) = U_{ss} + u(t)$ and $C(z,t) = C_{ss}(z) + c(z,t)$ leads now to the linearized boundary and distributed control problem
  \begin{align}
  \label{linear}
  \begin{split}
  c_t(z,t) = Dc_{zz}(z,t) - U_{ss} c_z(z,t) - C_{ss}'(z) u(t) - kc(z,t) \\
   Dc_z(0,t) - U_{ss}c(0,t) + (C_{in}-C_{ss}(0)) u(t)=0,  \qquad
  c_z(L,t)=0.
  \end{split}
  \end{align}
  The linearized output is $y(t) = c(L,t)$. In this case steady-state $C_{ss}(z)$ and transfer function $G(s) = y(s)/u(s)$
  of the linearized  equation can be computed formally using Maple. We obtain

  \begin{align}
C_{ss}(z)=C_{in}\,b\,\frac{(b-f)e^{\frac{f(1-{{z}/{L}})-{{bz}/{L}}}{2a}}-(b+f)e^{\frac{f({{z}/{L}}-1)-{{bz}/{L}} }{2a}}}{(-bf-2ak-b^2)e^{\frac{-f}{2a}}-(bf-2ak-b^2)e^{\frac{f}{2a}}}
\end{align}
%\begin{align}
%C_{ss_z}(z)=\frac{C_{in}(f+b)(f-b)}{2La}
%\frac{e^{\frac{(1-{\frac{z}{L}})f-{\frac{z}{L}} b}{2a}}-e^{\frac{({\frac{z}{L}}-1)f- \lambda b}{2a}}}{(-bf+2ac-b^2)e^{\frac{-f}{2a}}-(bf+2ac-b^2)e^{\frac{f}{2a}}}
%\end{align}
where  
\begin{equation}\label{eqabc} 
a=\frac{D}{L^2},\quad b=\frac{-U_{ss}}{L}, \quad f = \sqrt{b^2+4ak}.
\end{equation}
Then the transfer $u(s)\to c(z,s)$ is obtained analytically as $C_{in}\, \frac{P_1(z,s)}{P_2(z,s)}$,  where
\begin{align*}
P_1(z,s)=&\left[2La(s+k)\int_0^1 f_1(x)dx-L T_2\int_0^{\frac{z}{L}} f_2(x)dx-T_3(m-1)\right]e^{T_4(z)}+\\
&\left[LT_6\int_0^{\frac{z}{L}} f_1(x)dx+2La(s+k)\int_0^1f_2(x)dx+T_8(m-1)\right]e^{T_9(z)}+\\
&e^{T_5(z)}LT_2\int_1^{\frac{z}{L}} f_1(x)dx+e^{T_7(z)}LT_6\int_{\frac{z}{L}}^1f_2(x)dx\\
P_2(z,s)&=LT_1(T_2e^{\frac{T_1}{2a}}+T_6e^{\frac{-T_1}{2a}})
\end{align*}
with
\begin{align*}
& T_1=\sqrt{b^2+4a(s+k)},\quad T_2=bT_1+2a(-k-s)-b^2,\quad T_3=bT_1+4a(-k-s)-b^2,\\
&T_4(z)=\frac{{\frac{z}{L}} b+({\frac{z}{L}}-1)T_1}{-2a},\quad 
T_5(z)=\frac{-{\frac{z}{L}} b+({\frac{z}{L}}+1)T_1}{2a},\quad 
T_6=bT_1+2a(s+k)+b^2,\\ &
T_7(z)=\frac{{\frac{z}{L}} b+({\frac{z}{L}}+1)T_1}{2a},\quad
T_8=bT_1+4a(s+k)+b^2,\quad T_9(z)=\frac{-{\frac{z}{L}} b+({\frac{z}{L}}-1)T_1}{2a}\\
&f_1(z)=C_{ss_z}(z)e^{\frac{(b-T_1)z}{2aL}},\quad f_2(z)=C_{ss_z}(z)e^{\frac{(b+T_1)z}{2aL}},\quad m=\frac{C_{ss}(0)}{C_{in}}.
\end{align*}
The transfer $u \to y$ is then obtained at $z=L$.

  Adopting  numerical values
  $D= 1.05\, {\rm m}^2/{\rm min}$, $U_{ss} = 1.24\,{\rm m/min}$, $C_{\rm in}=0.5\,{\rm mol/m}^3$, $k=0.25\,{\rm m}^3{\rm /mol}$, and $L=6.36\,{\rm m}$
  from a study in \cite{fogler}, we can compare the the infinite-dimensional transfer function $G(s)$ with a finite-difference approximation
  $G_{\rm fd}(s)$. Figure \ref{compare} left shows the comparison of $G(s)$ and $G_{\rm fd}(s)$.
  
The scheme for synthesis is shown in Figure \ref{fig-scheme} and uses the filters in  Figure \ref{fig-filters} (right), which are defined as
\[
W_e(s)=\frac{0.00001 s + 5}{s+0.25}, \qquad W_n(s) = \frac{0.00125 s^2+0.00035 s +0.00005}{0.000025 s^2+0.007 s +1}, \qquad W_u=0.1.
\]
 The controller structure $\mathscr K_{\rm pi}$ includes SISO PI-controllers with two parameters $K(s)=k_p + \frac{k_i}{s}$, so $\x = (k_p,k_i)$.
 With the mixed performance-robustness channel $w=(r,n) \to z = (z_e,z_u)=(W_ee,W_uu)$ we have now defined the objective $\|T_{wz}(K(\x))\|_\infty$ of our
 problem, and according to section \ref{barrier}  the 
 objective $f$ 
 is complemented by the barrier function $c\|S(K(\x))\|_\infty$.  Note that  in the scheme 
 of Figure \ref{fig-scheme} the sensitivity function $S$ equals the unfiltered closed-loop
 transfer function $T_{re}$, so altogether with regard to the general form (\ref{hinfstruct}) in this study the channel
 $(r,n) \to (z_e,z_u,ce)$ is optimized, where $c=0.2$. The optimal solution obtained by algorithm \ref{algo1} is $\x^*=(k_i^*,k_p^*)=(9.93\cdot 10^{-5},7.13\cdot 10^2)$.

\begin{figure}[ht!]
\centerline{
\includegraphics[scale=1.0]{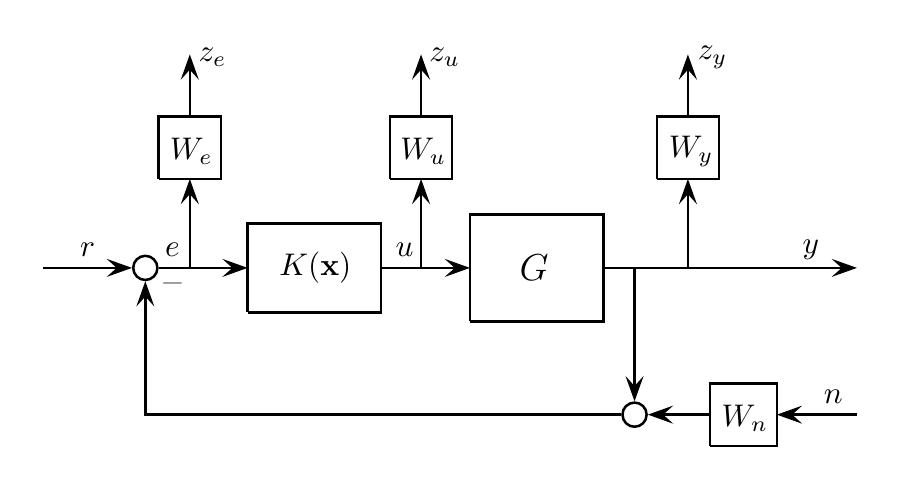}
}
\caption{\small Scheme for synthesis. In studies \ref{1D} and \ref{reactor}
the $H_\infty$-norm of the performance channel $(r,n)\to (z_e,z_u)$ is minimized, which assures that the system reacts to a set-point change $r$, 
and attenuates noise with
bandwidth specified by the filter $W_n$. Tracking $e$ is in the low-frequency range specified by $W_e$.   In study \ref{sect_cavity} the channel
$r \to (z_e,z_y)$ is optimized. \label{fig-scheme}}
\end{figure}
%\vspace*{2cm}

\begin{figure}[ht!]
\begin{center}
\includegraphics[scale=0.5]{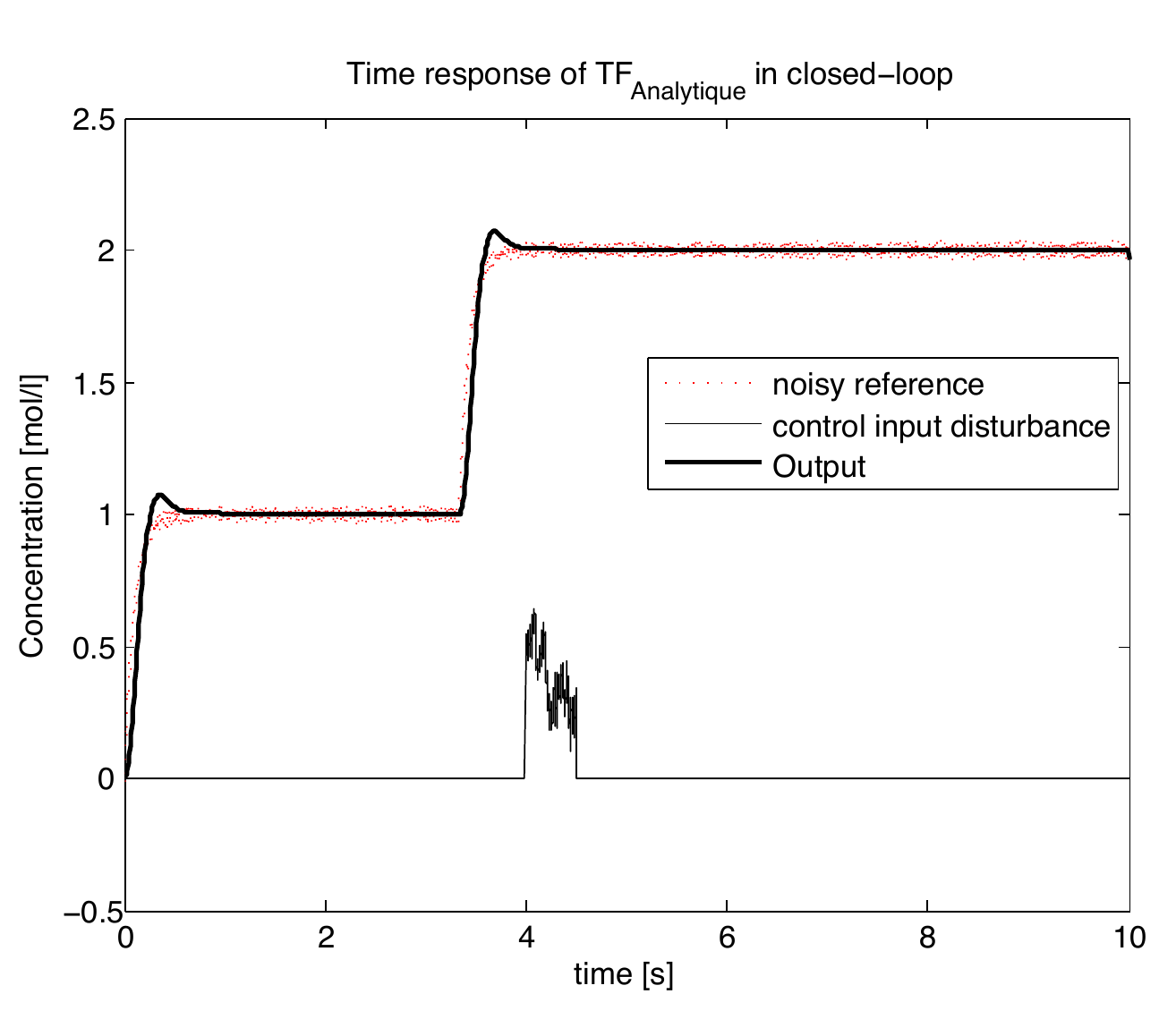}
\includegraphics[scale=0.5]{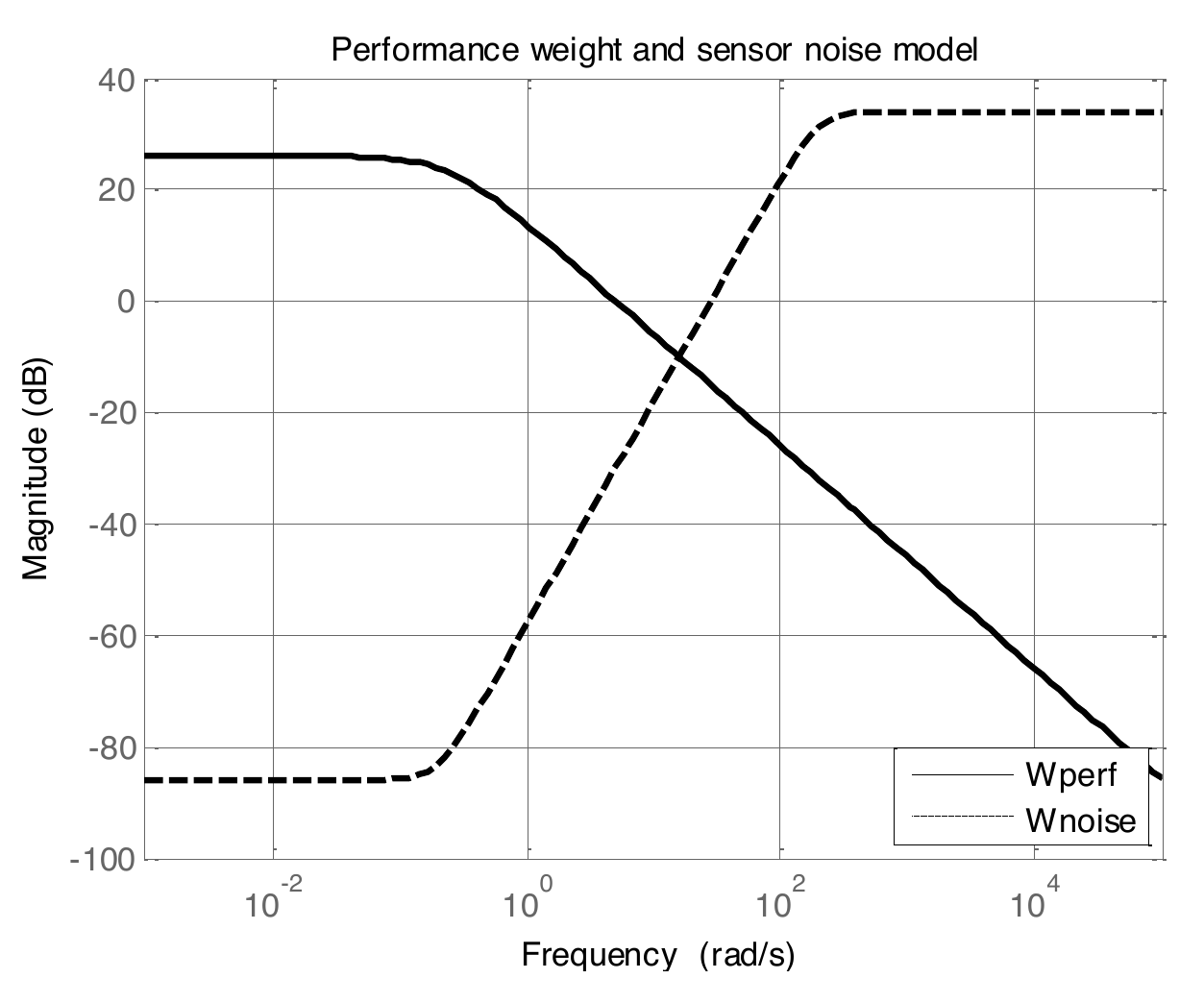}
\end{center}
\vspace*{-.2cm}
\caption{\small Left: Time response to a noisy reference signal of analytic closed loop. Right: Filter $W_e$  for synthesis and noise filter $W_n$. \label{fig-filters}}
\end{figure}

\begin{remark}
The functional analytic setup for (\ref{linear}) is as follows. On the Hilbert space $H=\mathcal L^2([0,L])$ define the differential operator
$\mathscr A = D\frac{d^2}{dz^2}-U_{ss} \frac{d}{dz} - k$ with domain $D(\mathscr A)=\{h\in H: h,\frac{d}{dz}h \, \mbox{a.c.}, \frac{d^2}{dz^2}h\in \mathcal L^2([0,L])\}$,
and the boundary control operator by $\mathscr P =\left( D \frac{d}{dz} - U_{ss}\right)/(C_{ss}(0)-C_{in})$ with domain
$D(\mathscr P)=\{h\in H: \frac{d}{dz} h \in \mathcal L^2([0,L]), h \mbox{ a.c.}\}$, so that $D(\mathscr A) \subset D(\mathscr P)$. One defines
$\alpha = \left( C_{ss}(0)-C_{in} \right) / \left( D-U_{ss}/2  \right)$ and the function $b(z)=\frac{\alpha}{2}(z+1)^2$, then
the multiplication operator $Bu = b(z)u$ satisfies $\mathscr P(Bu) = u$, and now we have a boundary control problem in the sense
of \cite[Def. 3.3.2]{zwart}, which can be brought to the form (\ref{system}).
\end{remark}

\subsection{Van de Vusse reactor}
\label{reactor}
The following more challenging study uses a
parallel {\em Van de Vusse process}
\begin{align*}
&A \stackrel{k_1}\longrightarrow B \stackrel{k_2}\longrightarrow C \\
&A + A \stackrel{k_3}\longrightarrow D
\end{align*}
which we want to operate in continuous mode in an isothermal reactor. In \cite{van_de_vusse}
Van de Vusse considers $A=$ cyclopentadiene,
$B=$ cyclopentenol, $C=$ cyclopentanediol,
$D=$ dicyclopentadiene. The desired output product is $B$, the inlet is the primary product $A$, whereas
$C,D$ are waste products. The reaction takes place in a steady-plug-flow cylindrical chemical reactor of length $L$, through which a fluid
with velocity $U$ is flowing, where ingredients are mixing axially  with dispersion coefficient $D$, while the now non-linear chemical reaction is described by the rates $k_i$. 
In ODE-based models the axial dispersion is often neglected, or replaced by a singular perturbation approach \cite{dochain}. 
Here we discuss the full non-linear model. For the functional analytic setup of the problem see again \cite{sano}.

Assuming 
radially homogeneous conditions in the tube, the system can be described by one  spatial dimension $z$, and 
the reaction for ingredients $A$ and $B$ is governed by the following diffusion-convection-reaction
system of parabolic PDEs:
\begin{align}
\begin{split}
\frac{\partial C_A(z,t)}{\partial t} &= D \frac{\partial^2 C_A(z,t)}{\partial z^2} - U \frac{\partial C_A(z,t)}{\partial z}
- k_1 C_A(z,t) - k_3 C_A^2(z,t) \\
\frac{\partial C_B(z,t)}{\partial t}&= D \frac{\partial^2 C_B(z,t)}{\partial z^2} - U\frac{\partial C_B(z,t)}{\partial z}
+ k_1 C_A(z,t) - k_2C_B(z,t)
\end{split}
\end{align}
for  $(z,t)\in [0,L]\times [0,\infty)$, 
with Danckwaerts and von Neumann boundary conditions
\begin{align}
\begin{split}
\label{bdry}
&D \frac{\partial C_A(0,t)}{\partial z} - U\left( C_A(0,t) - C_{Ain}  \right) = 0 , \quad D \frac{\partial C_B(0,t)}{\partial z} - U C_B(0,t)   = 0 \\
&\frac{\partial C_A(L,t)}{\partial z}= 0, \qquad \frac{\partial C_B(L,t)}{\partial z}=0
\end{split}
\end{align}
for all $t\in [0,\infty)$.   
The meaning of these boundary conditions at $z=0$ is that as soon as the feed enters the reactor at $z=0$, it will
be diluted by the axial mixing caused by the flow. At $z=L$ we have Neumann boundary conditions, which
simply require that the concentration stops changing at the point where the flow leaves the reactor. 

The goal of the study is to operate the reactor at a steady-state
flow $U_{ss}$ leading to a steady outflow $C_{Bss}(L)$ of product $B$ at the outlet of the reactor. This steady-state flow has to be controlled
by feedback, where we have the possibility to act on the velocity $U(t) = U_{ss} + u(t)$, and where
we use the deviation $y(t)=C_B(L,t) - C_{Bss}(L)$ from the steady-state production as our online measurement at the outlet. It is assumed that changes of the axial
flow velocity do not affect the axial dilution $D$ assumed constant. Control has to maintain a stable steady-state,
attenuate measurement noise and disturbances at the inflow, and enable the system to react to set-point changes in the flow velocity $U$.

Our procedure starts by computing the steady-state, which leads to solving the system of ODEs
\begin{align}
\begin{split}
&DC_{Ass}''(z)- U_{ss} C_{Ass}'(z)- k_1 C_{Ass}(z) - k_3 C_{Ass}^2(z) = 0 \\
&DC_{Bss}''(z) - U_{ss} C_{Bss}'(x) + k_1 C_{Ass}(z) - k_2 C_{Bss}(z) = 0 
\end{split}
\end{align}
with steady-state boundary conditions
\begin{align}
\label{bdry_ss}
\begin{split}
D  C_{Ass}'(0) - U_{ss} \left(  C_{Ass}(0)-C_{Ain}\right)  = 0, &\quad 
D C_{Bss}'(0) - U_{ss}   C_{Bss}(0) = 0, \\
C_{Ass}'(L)= 0 ,& \quad C_{Bss}'(L)=0.
\end{split}
\end{align}
This defines a mapping $U_{ss} \to (C_{Ass}(L),C_{Bss}(L))$, which 
allows us to see what flow $U_{ss}$ gives the largest output.
In the numerical study we fix $U_{ss} = 6.175e$-$3$, which leads to the solutions in Figure \ref{fig1}.

\begin{remark}
According to our strategy we assume that
$C_{Ass}(z), C_{Bss}(z)$ are computed with a very high precision, representing a quasi-analytic 
solution. Indeed, it is in principle possible to solve the steady-state system formally using Taylor series expansions, but since
the difference with a high precision numerical solution is marginal, we proceed with the numerical approach. 
\end{remark}

\begin{figure}[ht]
\centerline{
\includegraphics[scale=0.7]{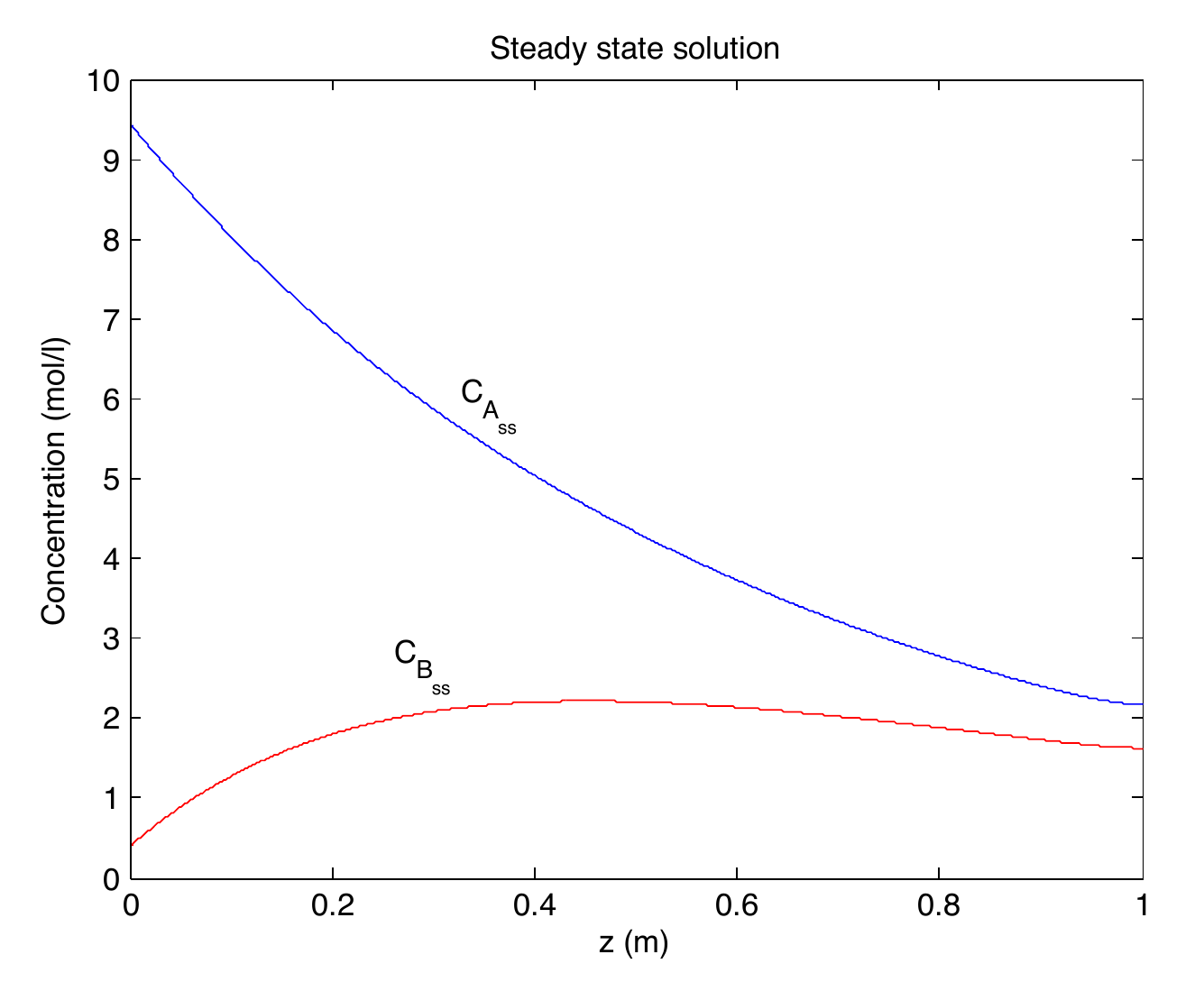}
}
\caption{Steady-state flow computed by Maple for $U_{ss}= 6.175e$-$3$. The parameters were {\tt maxmesh}=1e6, {\tt abserr} = 1e-7. The plots show
$C_{Ass}(z_i)$, $C_{Bss}(z_i)$ for 500 equidistant nodes on $[0,1]$. \label{fig1}}
\end{figure}

Once the steady-state is computed, we linearize the system by putting
$C_A(z,t) = C_{Ass}(z) + c_A(z,t)$, $C_B(z,t)=C_{Bss}(z) + c_B(z,t)$, $U(t) = U_{ss} + u(t)$ with off-sets $c_A,c_B,u$, which leads to
the linearized system 
\begin{align}
\label{linearized2}
\begin{split}
\frac{\partial c_A(z,t)}{\partial t} &= D \frac{\partial^2 c_A(z,t)}{\partial z^2} - U_{ss} \frac{\partial c_A(z,t)}{\partial z}
-(k_1 + 2k_3 C_{Ass}(z))c_A(z,t)   -\frac{\partial C_{Ass}(z)}{\partial z}\, u(t)\\
\frac{\partial c_B(z,t)}{\partial t} &= D \frac{\partial^2 c_B(z,t)}{\partial z^2} - U_{ss} \frac{\partial c_B(z,t)}{\partial z}
+k_1 c_A(z,t) - k_2c_B(z,t) - \frac{\partial C_{Bss}(z)}{\partial z} u(t)
\end{split}
\end{align}
with left boundary conditions
\begin{align}
\begin{split}
&D\frac{\partial}{\partial z} c_A(0,t)- U_{ss} c_A(0,t)  + (C_{Ain}  -C_{Ass}(0)) u(t) = 0 \\
&D\frac{\partial}{\partial z} c_B(0,t)- U_{ss} c_B(0,t)  - C_{Bss}(0) u(t) =0.
\end{split}
\end{align}
and right boundary conditions
\begin{align}
\frac{\partial c_A(L,t)}{\partial z}=0, \qquad \frac{\partial c_B(L,t)}{\partial z}=0.
\end{align}
The measured output is $y(t) = c_B(L,t)$.
The transfer function $G(s) = y(s)/u(s)$ is in principle also available analytically, 
but we continue with the high precision numerics strategy.
We first compute the
transfers  $c_A(z,s)/u(s)$ and $c_B(z,s)/u(s)$, which we obtain by Laplace transforming
the linearized system. This leads to the linear boundary value problem
\begin{align*}
D(c_A)_{zz}(z,s) - U_{ss} (c_A)_z(z,s) - (k_1+2k_3C_{Ass}(z) -s) c_A(z,s) - (C_{Ass})_z(z) u(s)=0\\
D(c_B)_{zz}(z,s) - U_{ss} (c_B)_z(z,s) + k_1 c_A(z,s) - (k_2+s)c_B(z,s) - (C_{Bss})_z(z)u(s)=0
\end{align*}
with Laplace transformed boundary conditions 
\begin{align*}
D(c_A)_z(0,s) - U_{ss} c_A(0,s) + (C_{Ain}-C_{Ass}(0)) u(s) = 0, \\
D(c_B)_{z}(0,s) - U_{ss}c_B(0,s) + (C_{Bin} - C_{Bss}(0)) u(s)= 0 \\
D(c_A)_z(L,s) = 0\\
D(c_B)_z(L,s) = 0
\end{align*}
Solving this boundary value problem  for fixed $s=j\omega$ with $u(s)\equiv 1$ gives the value $G(j\omega)$
of the transfer function $y(s) =c_B(L,s) = G(s) u(s)$. 
This is in fact the method presented in section \ref{illustration}. The magnitude of
$G(s)$ is shown in Figure \ref{sim} (right).

Note that for every fixed $s=j\omega$ we have to solve a static elliptic problem
associated with the dynamic equation (\ref{linearized2}), and we perform these computations with the finest scale
available, so that $G(s)$ is essentially lossless.
In Figure \!\ref{compare} it can be seen that in order to achieve the accuracy in $G(s)$ with a finite-difference discretization $G_{fd}(s)$
we would requires at least 2000 states. So in state-space we would have to perform synthesis
on a system of order 2000 to be sure that we do not lose information. This size is beyond existing synthesis techniques.
In the same vein, any approach based on system reduction would run the risk of losing information
in forming a transfer function based on a reduced model. 

\vspace*{.1cm}
\begin{center}
\begin{tabular}{||  c | c | c | c  || }
\hline\hline
constant & denomination & numerical value & unit \\
\hline\hline
$k_1$ & exchange rate $A\to B$ & $1.39\times 10^{-2}$ & ${\rm s}^{-1}$ \\
\hline
$k_2$ & exchange rate $B \to C$ & $2.78\times 10^{-2}$ & ${\rm s}^{-1}$ \\
\hline
$k_3$ & exchange rate $A+A \to D$ & $2.77 \times 10^{-4}$ & ${\rm l}/{\rm mol / s}$ \\
\hline
$D$ & axial dispersion coefficient& $3.33 \times 10^{-4}$ & ${\rm m}^2/{\rm s}$ \\
\hline
$U_{ss}$ & steady-state velocity & $6.175\times 10^{-3}$ & ${\rm m}/{\rm s}$ \\
\hline
$C_{Ain}$ & inlet concentration of component $A$ & $10.0$ & ${\rm mol}/{\rm l}$\\
\hline
$L$ & length of reactor& $1.0$ & ${\rm m}$\\
%\hline
%$u_{ss}$ & steady-state control & $0.0$ & ${\rm m\cdot mol}/{\rm s\cdot l}$ \\
\hline\hline
\end{tabular}
\end{center}
The spatiotemporally, spatially, and temporally  varying quantities are

\begin{center}
\begin{tabular}{||  c | c | c   || }
\hline\hline
Quantity & denomination & unit \\
\hline\hline
$C_A(z,t)$ & concentration of reactant $A$ & ${\rm mol}/{\rm l}$ \\
%\hline
%$N_A(x,t)$ & moles of component $A$ & ${\rm mol}$ \\
\hline
$C_B(z,t)$ & concentration of reactant $B$ & ${\rm mol}/{\rm l}$ \\
\hline
$C_{Ass}(z)$ & steady-state concentration of $A$ & ${\rm mol}/{\rm l}$ \\
\hline
$C_{Bss}(z)$ & steady-state concentration of $B$ & ${\rm mol}/{\rm l}$ \\
%\hline
%$C_{A0}(x)$ & initial concentration of $A$ & ${\rm mol}/{\rm l}$ \\
%\hline
%$C_{B0}(x)$ & initial concentration of $B$ & ${\rm mol}/{\rm l}$ \\
\hline
$U(t)$ & superficial velocity & ${\rm m}/{\rm s}$\\
%\hline
%$y(t)$ & measured output & ???? \\
\hline \hline
\end{tabular}
\end{center}

\begin{figure}[ht!]
\centerline{
\includegraphics[scale=0.55]{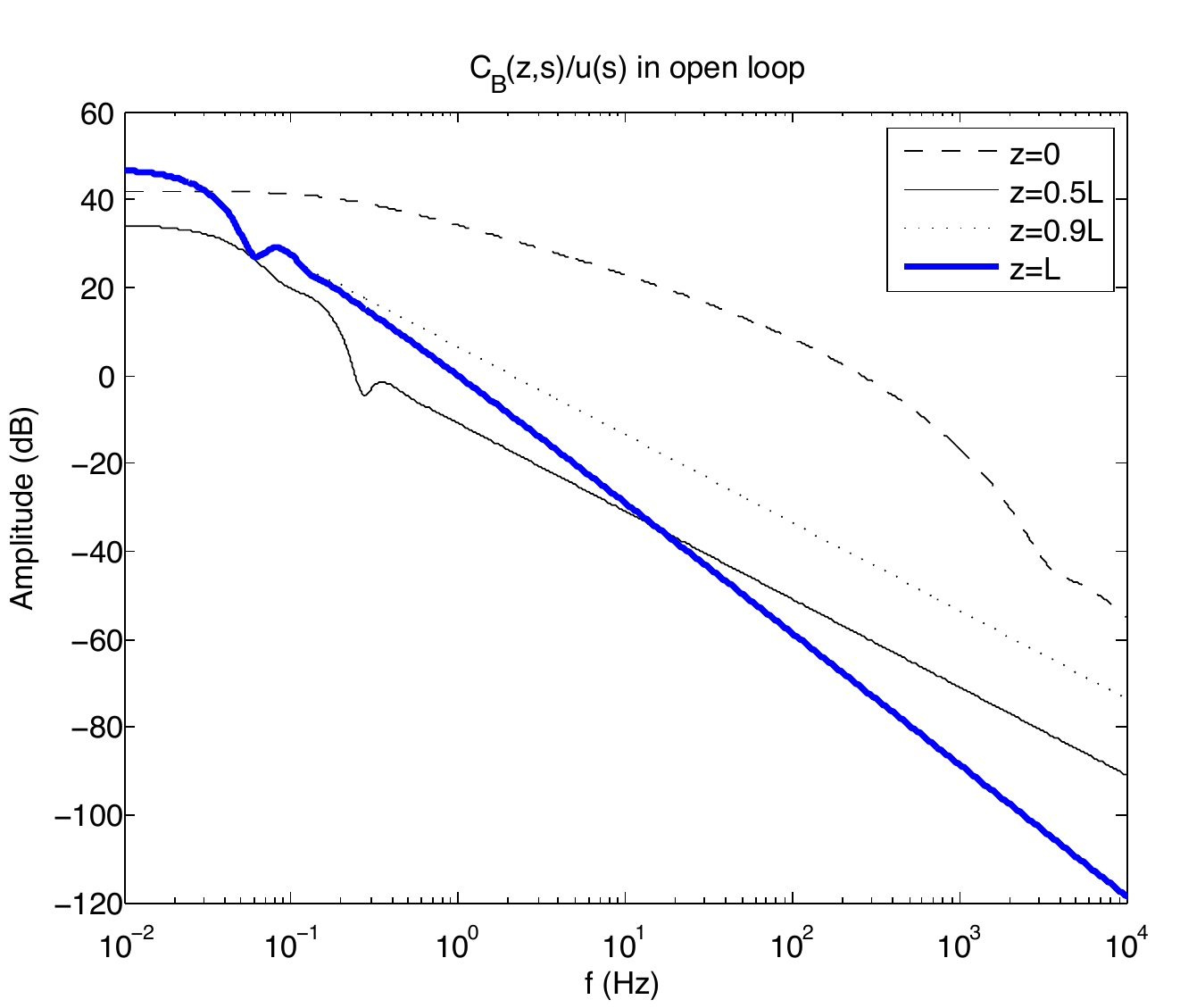}
\includegraphics[scale=0.5]{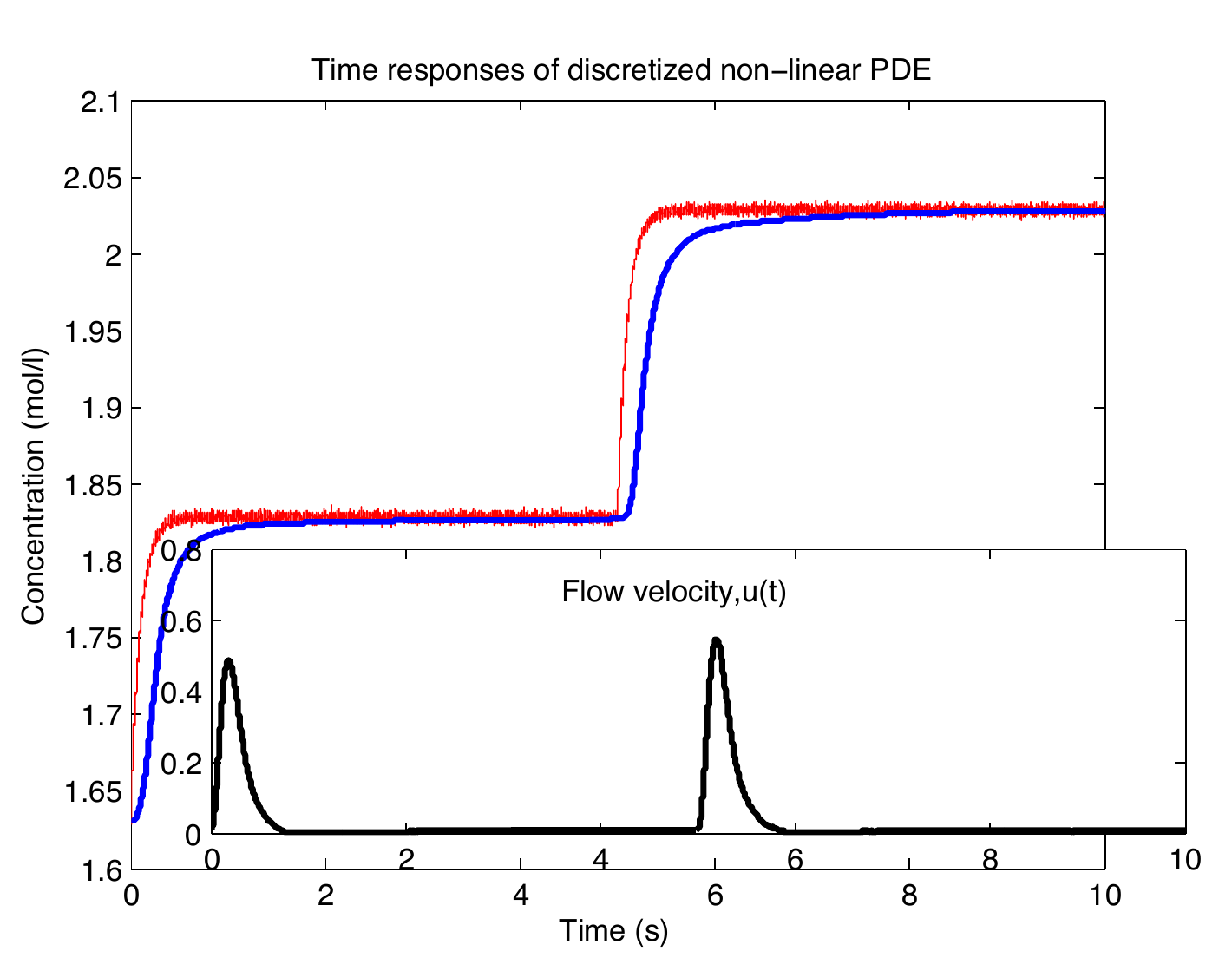}
}
\caption{\small Left: Comparison of transfer function magnitude for different positions of the sensor.  The blue curve ($z=L$) is the 
one chosen in the experiment.
Right: Response of closed-loop system to a noisy reference for optimal controller computed by algorithm \ref{algo3}. \label{sim}}
\end{figure}

\noindent
In the synthesis scheme of Figure \ref{fig-scheme} we use again the channel $(r,n) \to (z_e,z_u)$,  now with the filters
\[
W_e(s)=\frac{10^{-5}s + 1.502}{s+0.07509}, \quad
W_n(s)=\frac{0.00125s^2+0.00035s + 5\cdot 10^{-5}}{2.5\cdot 10^{-5} s^2 + 0.007s + 1}, \quad W_u = 0.1.
\]
Optimization is now
over the class $\mathscr K_3$ of third order controllers, which leads to a tunable vector $\x\in \mathbb R^{14}$, as the system matrix $A_K$ of the controller
(\ref{controller}) 
is parametrized in tridiagonal form. 
The channel is again complemented by the barrier $c\|S(K(\x))\|_\infty$, where $S=T_{re}$ and $c = 0.2$. 
The optimal $H_\infty$-controller $K(\x^*)$ with $\x^*\in \mathbb R^{14}$
computed by algorithm \ref{algo3} is obtained in the form (\ref{controller}) as
\[
A_K=\begin{bmatrix} -.8946 & -36.65 & 0 \\ -1.324 & -50.27 & -18.4 \\
0 & 82.2 & 12.14 \end{bmatrix},
B_K = \begin{bmatrix}  -8.32\\4.728 \\ 2.019 \end{bmatrix},
\begin{matrix}
C_K= \begin{bmatrix}  -1.839 & -3.129 & -9.019 \end{bmatrix}, \\
\vspace{.1cm}
D_K = -.08686.
\end{matrix}
\]
the optimal  gain being $f(\x^*) = 0.464$. The final number of nodes required for a certified result was $|\Omega_{\rm opt}| = 101$,  where 
one update of the grid  in step 6 of algorithm
\ref{algo3} was needed.
The study ends with a non-linear simulation of the optimal controller. In Figure \ref{sim} (right) the response if the nonlinear system to a noisy reference signal
is displayed.

\subsection{Cavity flow}
\label{sect_cavity}
We consider a challenging cavity flow study from \cite{ozbay}, where the infinite-dimensional transfer function
is available analytically and of the form
\[
G(s)=\frac{e^{-\tau_1s}}{p_2(s) + q_2(s)e^{-\tau_2 s} + c e^{-\tau_3 s}}
\]
with quadratic polynomials $p_2(s)$, $q_2(s)$ and delay parameters $\tau_i > 0$.  Figure \ref{cavity} (left) shows the magnitude
plot of $G$ in blue, indicating a large number of resonant peak frequencies.

\begin{figure}[h!]
\begin{center}
\includegraphics[scale=0.25]{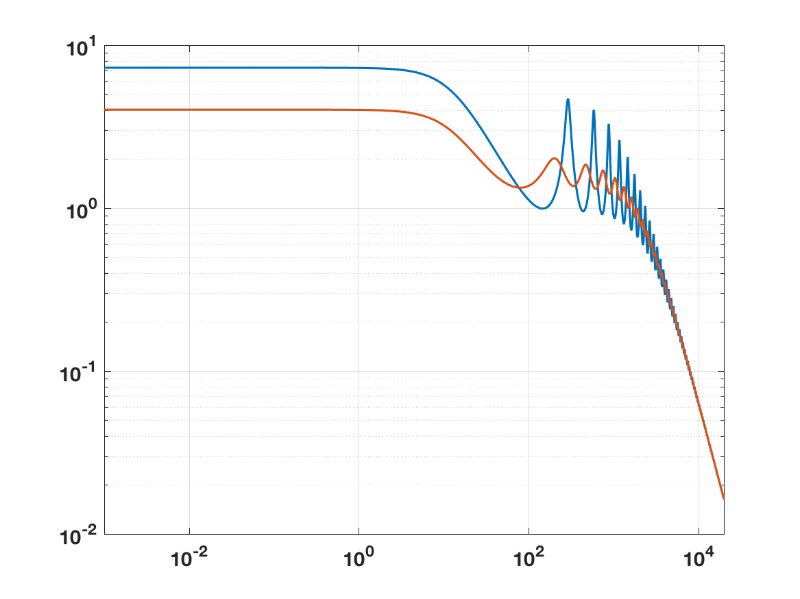}%{disturbanceTF.png}
\includegraphics[scale=0.9]{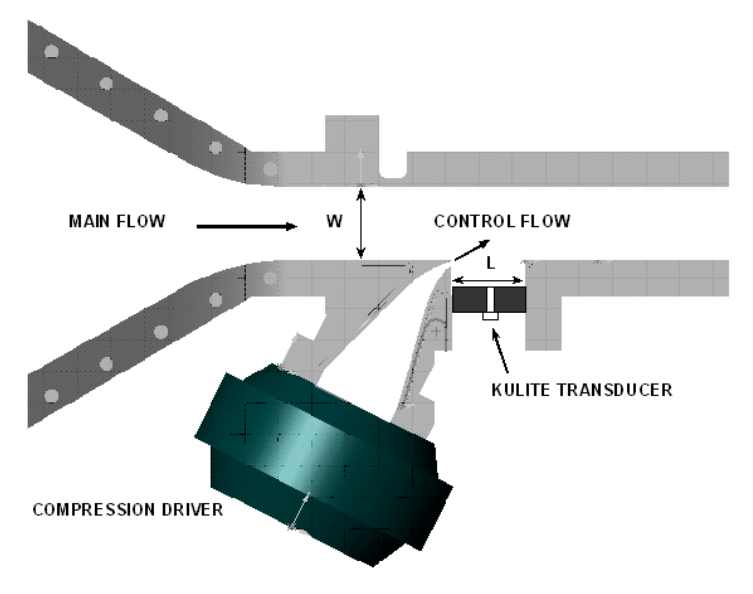}%{../data-driven/DATADRIVENpaper/nyquistCavityFlow.png}
\end{center}
\caption{\small Cavity flow study from \cite{ozbay}.  Left image shows magnitude of $G(j\omega)$ (blue), and of $GS$ in closed loop (red).  Sharp peaks are significantly reduced
by the synthesized control law. \label{cavity}}
\end{figure}

As $H_\infty$-objective we have chosen the channel $\|( W_1S,W_2T )\|_\infty$ with $S$ the closed-loop sensitivity function,
$T$ the complementary sensitivity functions, and with the frequency weighing filters
$W_1(s)=(0.01s+177.4)/(s+50.68)$, $W_2(s)=(100s+500)/(s+50000)$. Optimization (\ref{hinfstruct}) is over the class
$\mathscr K_2$ of order 2 controllers, which features 9 tunable parameters. The optimal controller $K(\x^*)$ with $\x^*\in \mathbb R^9$
computed by algorithm \ref{algo3} is given in transfer function form as
$K^*(s) = (0.5069 s^2 + 119.2s + 1419)/(s^2 + 308s + 1.276e04)$,  and achieves a gain of $\gamma^*= 1.937$, which
improves over the value 1.948 obtained in \cite{ozbay} using a coprime factorization approach.  The final grid size is $|\Omega_{\rm opt}| = 382$. 
The fact that the class $\mathscr K_2$ with which we achieve  $\gamma^*=1.937$  is much simpler than the infinite-dimensional
controller structure  used in \cite{ozbay} explains why
our novel approach represents a significant improvement in this study.

\section{Conclusion}
We have presented a novel method to compute $H_\infty$-controllers for infinite dimensional systems and in particular for boundary and distributed control
of PDEs. At the core our approach uses a non-smooth trust-region bundle
algorithm to solve a  frequency discretized version of the infinite-dimensional problem. The method was justified theoretically and tested numerically on a reaction-convection-diffusion
equation, on a Van de Vusse reactor, and for control of a cavity flow.   A convergence certificate for the non-smooth trust-region algorithm under 
Kiwiel's aggregation rule was proved, allowing to limit the number of cuts in the tangent program (\ref{tangent}) to any fixed number $N \geq 3$, and answering 
in the affirmative a question
left open in \cite{ANR:2016}.

\end{document}